\documentclass[english]{article}
\usepackage{layout}
\usepackage{babel}
\usepackage[OT1]{fontenc}
\usepackage{graphicx}
\usepackage{endnotes}
\usepackage{latexsym}
\usepackage{amsmath}
\usepackage{amssymb}
\usepackage{amsthm}
\usepackage[mathscr]{eucal}
\usepackage{amsbsy}
\usepackage{epsfig}

\newtheorem{theorem}{Theorem}
\newtheorem{prop}{Proposition}
\newtheorem{lemma}{Lemma}
\newtheorem{corollary}{Corollary}

\newtheorem{preremark}{Remark}

\def\RR{{\Bbb R}}

\def\W2{W^{1,2}({\cal O}(M))}

\def\RR{\mathbb{R}}

\def\NN{\mathbb{N}}

\def\1half{\frac{1}{2}}

\newcommand{\abs}[1]{\left| #1 \right|}
\newcommand{\Hprod}[2]{\langle #1 , #2 \rangle_{\mathcal{H}}}

\newcommand{\Hnorm}[1]{\left\| #1 \right\|_{\mathcal{H}}}
\newcommand{\HHnorm}[1]{\left\| #1 \right\|_{\mathcal{H}^{\otimes 2}}}

\newcommand{\Lprod}[3]{\left\langle #1 , #2 \right\rangle_{L^2_{#3}}}

\newcommand{\Linfnorm}[1]{\left\| #1 \right\|_{\infty}}
\newcommand{\Lqnorm}[2]{\left\| #1 \right\|_{L^{2}_{#2}}}
\newcommand{\Lqprod}[3]{\left\langle #1 , #2 \right\rangle_{L^{2}_{#3}}}
\newcommand{\opnorm}[1]{\left\|#1\right\|_{\text{op}}}

\newcommand{\dint}{\int\hspace{-.2cm}\int}

\begin{document}

\title{On the second order Poincar$\acute{\text{e}}$ inequality and CLTs on Wiener-Poisson space}
%
%
%
\author{Juan Jos$\acute{\text{e}}$ V$\acute{\text{\i}}$quez R}

%
%
%
%
%
\maketitle

\begin{abstract}
An upper bound for the Wasserstein distance is provided in the general framework of the Wiener-Poisson space. Is obtained from this bound a second order Poincar$\acute{\text{e}}$-type inequality which is useful in terms of computations. For completeness sake, is made a survey of these results on the Wiener space, the Poisson space, and the Wiener-Poisson space, and showed several applications to central limit theorems with relevant examples: linear functionals of Gaussian subordinated fields (where the subordinated field can be processes like fractional Brownian motion or the solution of the Ornstein-Uhlenbeck SDE driven by fractional Brownian motion), Poisson functionals in the first Poisson chaos restricted to ``small'' jumps (particularly fractional L\'evy processes) and the product of two Ornstein-Uhlenbeck processes (one in the Wiener space and the other in the Poisson space). Also, are obtained bounds for their rate of convergence to normality.
\end{abstract}

\section{Introduction}

In recent years many papers have looked at combining Stein's method with Malliavin calculus in order to uncover new tools for proving varios central limit theorems (CLTs). For example, I. Nourdin and G. Peccati derived an upper bound for the Wasserstein (Kantorovich) distance (and other distances) on the Wiener space using Stein's equation, \cite{Nourdin0}. Later, the same authors along with G. Reinert derived a second order Poincar$\acute{\text{e}}$(-type) inequality which is useful (in terms of computations) for proving CLTs, and which in fact, can be seen as a quantitative extension of the Stein's method from which upper bounds for the rate of convergence to normality can be found, \cite{Nourdin}. The first two authors with A. R$\acute{\text{e}}$veillac extended these results to the multidimensional case, \cite{Nourdinmult}. In \cite{Peccati}, G. Peccati, J. L. Sol$\acute{\text{e}}$, M. S. Taqqu and F. Utzet, were able to find an upper bound, similar to the one in \cite{Nourdin0}, for the Wasserstein distance in the Poisson space. G. Peccati and C. Zheng succeeded in extending this to the multi-dimensional case in \cite{Peccati2}. All these works are important as they give quantitative tools for computing whether a random variable converges to normality or not, and if so, its rate of convergence.

\bigskip

The upper bound inequality can be written as,
$$d_W(F,N)\leq E\abs{1-\Hprod{\boldsymbol{D}F}{-\boldsymbol{D}L^{-1}F}}+E\bigl[\Hprod{\abs{x(\boldsymbol{D}F)^2}}{\abs{\boldsymbol{D}L^{-1}F}}\bigr]$$
where $d_W$ is the Wasserstein distance, $N\sim\mathcal{N}(0,1)$, $\boldsymbol{D}$ is the Malliavin derivative, $L^{-1}$ is the inverse of the infinitesimal generator of the Ornstein-Uhlenbeck (O-U) semigroup, and $\mathcal{H}$ is the underlying Hilbert space. In the case of the Wiener space upper bound, $\boldsymbol{D}$ is the Malliavin derivative defined in that space so this inequality holds even when the underlying Hilbert space is not $L^2_{\mu}$. Also, since there are no jumps here, the second term on the right disappears ($x$ is the size of the jump). In the Poisson space case, since we do not (as yet) have a Malliavin calculus theory developed for a general abstract Hilbert space, the underlying Hilbert space must be $\mathcal{H}=L^2_{\mu}$. A question naturally arises - can this be done for a general L$\acute{\text{e}}$vy process; that is, is this upper bound achievable in a mixed space, the Wiener-Poisson space? The main difficulty in answering this question is that in the Wiener-Poisson space we don't have decomposition in orthogonal polynomials (unlike with Hermite polynomials in the Wiener space, see \cite{NunnoOrthpol} for a complete explanation), and results like the equivalence between the Mehler semigroup and the Ornstein-Uhlenbeck semigroup. So, to overcome these shortcoming, will be necessary to deduce new formulas that will allow us to follow the ideas developed in \cite{Nourdin0} and \cite{Nourdin}, and recover their results for the Wiener-Poisson space. Will be shown that this bound still holds even when both spaces are involved.

\bigskip

Before get into the details, some notation: The measure $\mu$ on the underlying Hilbert space $L^2_{\mu}$ is defined by the underlying L$\acute{\text{e}}$vy process, that is, let $L_t$ be a L$\acute{\text{e}}$vy process ($L_t$ has stationary and independent increments, is continuous in probability and $X_0=0$, with $E[L_1^2]<\infty$) with L$\acute{\text{e}}$vy-triplet given by $(0,\sigma^2,\nu)$, where $\nu$ is the L$\acute{\text{e}}$vy measure, then $\dint_{\RR^+\times\RR}f(z) d\mu(z)=\sigma^2\int_{\RR^+}f(t,0) dt+\dint_{\RR^+\times\RR_0}f(t,x) x^2dtd\nu(x)$, where $\RR_0=\RR-\{0\}$. On the other hand, in order to define a Malliavin derivative in the Wiener-Poisson space it is sufficient to have a chaos decomposition of the space $L^2(\Omega)$. This was achieved in \cite{Ito} by K. It$\hat{\text{o}}$, so any random variable in $L^2(\Omega)$ has a projection on the $q^{\text{th}}$ chaos given by $I_q(f_q)$, where $f_q$ is a symmetric function in $L^2_{\mu}$. Also, when working in the Wiener-Poisson space, the Malliavin derivative can be regarded in terms of ``directions'', i.e., we can think of it as the derivative in the Wiener direction or the derivative in the Poisson direction. The fact that this can be done in this way was proved in \cite{Josep} by J. L. Sol$\acute{\text{e}}$, F. Utzet and J. Vives (a quick review of the theory is given below). They showed that the Malliavin derivative with parameter $z\in\RR^+\times\RR$ can be split in two cases, when $z=(t,0)$ and when $z(t,x)$ with $x\neq0$. The first case will be the derivative in the Wiener direction (intuitively because there are no jumps when $x=0$), and the second will be the derivative in the Poisson direction. A distinction between the Malliavin calculus in the Wiener space or the Poisson space and in this Wiener-Poisson space is the need to define two subspaces of $L^2(\Omega)$: one where the Malliavin derivative in the Wiener direction coincides with the usual Malliavin derivative in the Wiener space and is well defined, and another where the Malliavin derivative in the Poisson direction is well defined. These subspaces are denoted by Dom$\boldsymbol{D}^W$ and Dom$\boldsymbol{D}^J$.

\bigskip

\noindent {\bf Theorem 1.}{\em (Main result 1: upper bound)\\
Let $N\sim\mathcal{N}(0,1)$ and let $F\in$Dom$\boldsymbol{D}^W\cap$Dom$\boldsymbol{D}^J$ be such that $E[F]=0$. Then,}
\begin{align*}
d_W(F,N)\leq E\abs{1-\Lqprod{\boldsymbol{D}F}{-\boldsymbol{D}L^{-1}F}{\mu}}+E\bigl[\Lqprod{\abs{x(\boldsymbol{D}F)^2}}{\abs{\boldsymbol{D}L^{-1}F}}{\mu}\bigr]
\end{align*}
As mentioned above, in \cite{Nourdin}, the authors used the respective Wiener space upper bound to deduce a second order Poincar$\acute{\text{e}}$ inequality. In this sense, our second main result is the second order Poincar$\acute{\text{e}}$ inequality in the Wiener-Poisson space derived from the above upper bound. This version is weaker than the one in the Wiener space because of the scarcity of results in this space, that would enable us to prove it for functionals in Dom$\boldsymbol{D}^W\cap$Dom$\boldsymbol{D}^J$; nevertheless, is possible to state it for functionals that lie in one specific It$\hat{\text{o}}$-chaos.

\bigskip

\noindent {\bf Corollary 1.}{\em (Main result 2: Second order Poincar$\acute{\text{e}}$-type inequality)\\
Fix $q\in\NN$ and let $F=I_q(f)$ with $E[F]=\mu$ and Var$[F]=\sigma^2$, and assume that $N\sim\mathcal{N}(\mu,\sigma^2)$. Then,}
\begin{align*}
\hspace{-.5cm} d_W(F,N)&\leq \frac{\sqrt{2}}{q\sigma^2}\biggl(2E\bigl[\opnorm{\boldsymbol{D}^2F}^4\bigr]^{\frac{1}{4}}E\bigl[\Lqnorm{\boldsymbol{D}F}{\mu}^4\bigr]^{\frac{1}{4}}+E\biggl[\Lqnorm{\Lqprod{x}{(\boldsymbol{D}^2F)^2}{\mu}}{\mu}^2\biggr]^{\frac{1}{2}}\biggr)+\frac{1}{q\sigma^3}E\bigl[\Lqprod{\abs{x}}{\abs{\boldsymbol{D}F}^3}{\mu}\bigr]\\
&\leq\frac{\sqrt{2}}{q\sigma^2}\biggl(2E\bigl[\Lqnorm{\boldsymbol{D}^2F\otimes_1\boldsymbol{D}^2F}{\mu^{\otimes2}}^2\bigr]^{\frac{1}{4}}E\bigl[\Lqnorm{\boldsymbol{D}F}{\mu}^4\bigr]^{\frac{1}{4}}+E\biggl[\Lqnorm{\Lqprod{x}{(\boldsymbol{D}^2F)^2}{\mu}}{\mu}^2\biggr]^{\frac{1}{2}}\biggr)+\frac{1}{q\sigma^3}E\bigl[\Lqprod{\abs{x}}{\abs{\boldsymbol{D}F}^3}{\mu}\bigr]
\end{align*}

\bigskip

Thanks to these (Wiener, Poisson and Wiener-Poisson) upper bounds and second order Poincar$\acute{\text{e}}$ inequalities, many CLTs can be proved and generalizations made. In this paper, these bounds are reviewed for each space, showing their importance by giving applications with relevant examples. In the Wiener space case the second order Poincar$\acute{\text{e}}$ inequality is used to prove CLTs for linear functionals of Gaussian-subordinated fields when the decay rate of the covariance function of the underlying Gaussian process satisfies certain conditions. These CLTs are applied to the important cases where the underlying Gaussian process is either the fractional Brownian motion or the fractional-driven Ornstein-Uhlenbeck process, with $H\in\bigl(0,\frac{1}{2}\bigr)\cup\bigl(\frac{1}{2},1\bigr)$.

\bigskip

In the Poisson space case, the respective upper bound is used to prove that the small jumps process (jumps with length less than or equal to $\epsilon$) of a Poisson functional process with infinitely many jumps goes to a normal random variable when $\epsilon$ goes to zero. Furthermore, a remarkable extension, of the known result (proved in \cite{Asmussen}) which states that the small jumps process of a L$\acute{\text{e}}$vy process can be approximated by Brownian motion as $\epsilon$ goes to zero, is proved. It is extended to Poisson functionals $\bigl(I_1(f)\bigr)$ and showed that the small jumps process of this functional can be approximated by a Gaussian functional with the same kernel $f$ as $\epsilon$ goes to zero. Then is applyed this result to show that in order to simulate a fractional (pure jump) L$\acute{\text{e}}$vy process (FLP), it is sufficient to simulate a process with finitely many jumps plus an independent fractional Brownian motion (FBM).

\bigskip

Finally, the second order Poincar$\acute{\text{e}}$(-type) inequality, developed in this paper, is used to prove that the time average of the product of a Wiener Ornstein-Uhlenbeck process with a Poisson Ornstein-Uhlenbeck process converges to a normal random variable as time goes to infinity. This example highlights the importance of the inequality in the Wiener-Poisson space, since it cannot be achieved by the upper bounds in the Wiener or Poisson space. An estimate of the rate of convergence to normality is obtained in the examples where the second order Poincar inequalities is used.

\bigskip

The paper is organized as follows: In Section 2 is recalled the basic tools of Malliavin calculus on the Wiener space as well as the second order Poincar$\acute{\text{e}}$ inequality. Then is stated the Malliavin calculus results for the Wiener-Poisson space and explain how these tools have been extended. In Section 3 the theory developed in \cite{Nourdin} and \cite{Nourdin0} is reproduced using Malliavin calculus theory for the Wiener-Poisson space, and with this framework is stated a ``L$\acute{\text{e}}$vy version'' of the second order Poincar$\acute{\text{e}}$ inequality. Section 4 is dedicated to going over the inequalities for the Wiener, Poisson and Wiener-Poisson spaces. In the Wiener space case, a result proved in \cite{Nourdin} concerning CLTs of linear functionals of Gaussian-subordinated fields is extended. In the Poisson space case, a result on the simulation of small jumps for processes with infinitely many jumps is given. Finally, an example of application of the second order Poincar$\acute{\text{e}}$ inequality in the Wiener-Poisson space is showed.

\section{Preliminaries}

As mentioned above, the most important tool for proving CLTs on the Wiener space is the upper bound developed in \cite{Nourdin}. This requires various Malliavin calculus results on the Wiener space (Malliavin derivative, contraction of order $r$, Mehler formula, etc) which are extensively studied and explained in \cite{Nualart}. For the sake of the completeness, the basic tools from Malliavin calculus in the Wiener space are reviewed and then is defined the Malliavin calculus in the Wiener-Poisson setting, both needed in this paper.

\subsection{Malliavin Calculus on Wiener space}

Let $\mathcal{H}$ be a real separable Hilbert space. Assume there is a probability space $(\Omega,\mathcal{F},\mathbb{P})$ where $W:=\{W(h)/\ h\in\mathcal{H}\}$ is an isonormal Gaussian process, that is, $W$ is a centered Gaussian family s.t. $E[W(h_1)W(h_2)]=\Hprod{h_1}{h_2}$. Choose $\mathcal{F}$ to be the $\sigma$-algebra generated by $W$. Let $H_q$ be the $q^{\text{th}}$ Hermite polynomial, $H_q(x)=(-1)^qe^{\frac{x^2}{2}}\frac{\partial^q}{\partial x^q}\bigl(e^{-\frac{x^2}{2}}\bigr)$, and define the $q^{\text{th}}$ Wiener chaos of $W$ (denoted by $\mathbb{H}_q$) as the subspace of $L^2(\Omega):=L^2(\Omega,\mathcal{F},\mathbb{P})$ generated by $\{H_q(W(h))/\ h\in\mathcal{H},\Hnorm{h}=1\}$. Is important to underline that $L^2(\Omega)$ can be decomposed (Wiener chaos expansion) into an infinite orthogonal sum of the spaces $\mathbb{H}_q$: $L^2(\Omega)=\oplus_{q=0}^\infty \mathbb{H}_q$
\begin{preremark}\label{chaosintegral}
In the case where $\mathcal{H}=L^2_\mu$, for any $F\in L^2(\Omega)$,
\begin{align}\label{chaosrep0}
F=\sum_{q=0}^\infty I_q(f_q)
\end{align}
where $I_q$ is the $q^{\text{th}}$ multiple stochastic integral, $f_0=E[F]$, $I_0$ is the identity mapping on constants and $f_q\in L^2_{\mu^{\otimes q}}$ are symmetric functions uniquely determined by $F$.
\end{preremark}
Let $\mathcal{S}$ be the class of smooth random variables, i.e., if $F\in\mathcal{S}$ then there exists a function $\phi\in C^\infty(\RR^n)$ s.t. $\frac{\partial^k \phi}{\partial x_i^k}(x)$ has polynomial growth for all $k=0,1,2,3,\dots$ and $F=\phi(W(h_1),\dots,W(h_n))$, $h_i\in\mathcal{H}$. The Malliavin derivative of $F\in \mathcal{S}$ with respect to $W$ is the element of $L^2(\Omega,\mathcal{H})$ defined as
\begin{align}\label{malliavinderivative}
\boldsymbol{D}F=\sum_{i=1}^n\frac{\partial \phi}{\partial x_i}\bigl(W(h_1),\dots,W(h_n)\bigr)h_i
\end{align}
In particular $\boldsymbol{D}W(h)=h$ for every $h\in\mathcal{H}$. Notice that in this particular case we have an explicit relation between the covariance of $W$ and the inner product of the Malliavin derivate, $\text{Cov}[W(h_1)W(h_2)]=E[W(h_1)W(h_2)]=\Hprod{h_1}{h_2}=\Hprod{\boldsymbol{D}W(h_1)}{\boldsymbol{D}W(h_2)}$.
\begin{preremark}\label{hilbertderivative}
In the case of a centered stationary Gaussian process, $X_t$, the Hilbert space can be chosen in the following way:\\
Consider the inner product $\Hprod{1_{[0,t]}}{1_{[0,s]}}=\text{Cov}[X_tX_s]=C(t-s)$ and take the Hilbert space $\mathcal{H}$ as the closure of the set of step functions on $\RR$ with respect to this inner product. With this Hilbert space one has that $X_t=W(1_{[0,t]})$ and $\boldsymbol{D}X_t=1_{[0,t]}$ where $\boldsymbol{D}$ is the Malliavin derivative.
\end{preremark}
Since the Malliavin derivative verifies the chain rule, we have $\boldsymbol{D}f(F)=f'(F)\boldsymbol{D}F$, for any $f:\RR\rightarrow\RR$ of class $\mathcal{C}^1$ with bounded derivative (is also true for functions which are only a.e. differentiable, but with the assumption that $F$ has absolutely continous law). The second Malliavin derivative (denoted by $\boldsymbol{D}^2$) can be define recursively. For $k\geq1$ and $p\geq 1$, $\mathbb{D}^{k,p}$ denotes the closure of $\mathcal{S}$ with respect to the norm $\left\|\cdot\right\|_{k,p}$ defined by
$$\left\|F\right\|_{k,p}^p=E\bigl[\abs{F}^p\bigr]+\sum_{i=1}^kE\bigl[\left\|\boldsymbol{D}^iF\right\|^p_{\mathcal{H}^{\otimes i}}\bigr]$$
Consider now an orthonormal system in $\mathcal{H}$ denoted by $\{e_k/\ k\geq1\}$. Then, given elements $\phi\in\mathcal{H}^{\otimes k_1},\psi\in\mathcal{H}^{\otimes k_2}$, the contraction of order $r\leq\min\{k_1,k_2\}$ is the element of $\mathcal{H}^{\otimes (k_1+k_2-2r)}$ defined by
$$\phi\otimes_r\psi=\sum_{i_1,\dots,i_r=1}^\infty \left\langle \phi,e_{i_1}\otimes\cdots e_{i_r}\right\rangle_{\mathcal{H}^{\otimes r}}\left\langle \psi,e_{i_1}\otimes\cdots e_{i_r}\right\rangle_{\mathcal{H}^{\otimes r}}$$
In particular, $\phi\otimes_r\psi=\left\langle \phi,\psi\right\rangle_{\mathcal{H}^{\otimes r}}$ when $k_1=k_2=r$.
\begin{preremark}\label{prodwiener}
Again, in the white noise framework (when $\mathcal{H}=L^2_\mu$), for symmetric functions $\phi\in L^2_{\mu^{
\otimes k_1}}$, $\psi\in L^2_{\mu^{\otimes k_2}}$, the contraction is given by the integration of the first $r$ variables, i.e., $\phi\otimes_r\psi=\Lqprod{\phi}{\psi}{\mu^{\otimes r}}$. Also, we have a formula for the product of stochastic integrals:\footnote{$a\wedge b = \min\{a,b\}$ and $a\vee b = \max\{a,b\}$} $I_p(f)I_q(g)=\sum_{r=0}^{p\wedge q}r! {p \choose r}{q \choose r}I_{p+q-2r}(f\otimes_r g)$
\end{preremark}

Define the divergence operator $\delta$ as the adjoint of the operator $\boldsymbol{D}$, so if $F\in$Dom$\delta$ then $\delta(F)\in L^2(\Omega)$ and $E[\delta(F)G]=E[\Hprod{\boldsymbol{D}G}{F}]$.
\begin{preremark}\label{skorohod1}
When $\mathcal{H}=L^2_\mu$, the divergence operator $\delta$ is called the Skorohod integral. It is a creation operator in the sense that for all $F\in$Dom$\delta\subset L^2_{\mu\times\mathbb{P}}(T\times\Omega)$ with chaos representation $F(t)=\sum_{q=0}^\infty I_q\bigl(f_q(t,\cdot)\bigr)$ ($f_q\in L^2_{\mu^{\otimes(q+1)}}$ are symmetric functions in the last $q$ variables), $\delta(F)=\sum_{q=0}^\infty I_{q+1}(\widetilde{f}_q)$.\footnote{$\widetilde{f}$ is the symmetrizatoin of $f$, i.e., $\tilde{f}(z_1,\dots,z_q)=\frac{1}{q!}\sum_\sigma f(z_{\sigma(1)},\dots,z_{\sigma(q)})$}
\end{preremark}
For all $F\in L^2(\Omega)$ denote by $J_qF$ the projection of $F$ in the $q^{\text{th}}$ chaos. Then, the Ornstein-Uhlenbeck semigroup is the family of contraction operators $\{T_t\ /\ t\geq0\}$ on $L^2(\Omega)$ defined by $T_tF=\sum_{q=0}^\infty e^{-qt}J_qF$. Using Mehler's formula we can find an equivalence between Mehler's semigroup and the O-U semigroup. More formally, take $W'$ as an independent copy of $W$ defining $(W,W')$ on the product probability space $(\Omega\times\Omega,\mathcal{F}\otimes\mathcal{F}',\mathbb{P}\times\mathbb{P})$. Each $F\in L^2(\Omega)$ can be regarded as measurable map $F(W)$ from $\RR^{\mathcal{H}}$ to $\RR$ determined $\mathbb{P}\circ W^{-1}$-a.s. such that $T_tF=E'\bigl[F\bigl(e^{-t}W+\sqrt{1-e^{-2t}}W'\bigr)\bigr]$. The infinitesimal generator for this semigroup (denoted by $L$) is given by $LF=\sum_{q=0}^\infty-qJ_qF$ and Dom$L=\mathbb{D}^{2,2}$=Dom$(\delta\boldsymbol{D})$. It can be proved that, for $F\in$Dom$L$, $\delta\boldsymbol{D}F=-LF$. The pseudo-inverse of this operator (denoted by $L^{-1}$) is given by $L^{-1}F=\sum_{q=1}^\infty\frac{-1}{q}J_qF$, and is such that $L^{-1}F\in$Dom$L$ and $LL^{-1}F=F-E[F]$ for any $F\in L^2(\Omega)$.

The use of Hermite polynomials is extremely important in this setting. The relationship between Hermite polynomials and Gaussian random variables is the following. Let $Z_1,Z_2$ be two random variables with joint Gaussian distribution such that $E[Z_1^2]=E[Z_2^2]=1$ and $E[Z_1]=E[Z_2]=0$. Then for all $q,p\geq0$,
\begin{align}\label{normalhermite}
E\bigl[H_p(Z_1)H_q(Z_2)\bigr]=
\begin{cases}
q!\bigl(E[Z_1Z_2]\bigr)^q & \text{if } q=p\\
0 & \text{if } q\neq p
\end{cases}
\end{align}
On the other hand, it is possible to expand a $\mathcal{C}^2$ function $f:\RR\rightarrow\RR$ in terms of Hermite polynomials, that is,
\begin{align}\label{decomp}
f(x)=E[f(Z)]+\sum_{q=1}^\infty c_qH_q(x)
\end{align}
where the real numbers $c_q$ are given by $c_qq!=E[f(Z)H_q(Z)]$ and $Z\sim\mathcal{N}(0,1)$.
\begin{preremark}\label{hermiteintegral}
Notice that in the white noise case we have the relationship $H_q\bigl(W(h)\bigr)=\int_{T^q}h^{\otimes q}dW_{t_1}\cdots dW_{t_q}=I_q\bigl(h^{\otimes q}\bigr)$ so the decomposition (\ref{decomp}) of $f(W(h))$ can be regarded as
$$f(W(h))=\sum_{q=0}^\infty c_qI_q\bigl(h^{\otimes q}\bigr)$$
\end{preremark}

With (\ref{normalhermite}) and (\ref{decomp}) we are able to compute the covariance of a real function $f$ in the following way,
\begin{align}\label{covf}
\text{Cov}[f(Z_1)f(Z_2)]=\sum_{p,q=1}^\infty c_pc_qE\bigl[H_p(Z_1)H_q(Z_2)\bigr]=\sum_{q=1}^\infty c_q^2q!\bigl(E[Z_1Z_2]\bigr)^q
\end{align}

With this background, is possible now to state the main tool developed in \cite{Nourdin} and used in this paper.
\begin{lemma}\label{nourdinineq}
Let $F\in\mathbb{D}^{2,4}$ with $E[F]=\mu$ and $\text{Var}[F]=\sigma^2$. Assume that $N\sim\mathcal{N}(\mu,\sigma^2)$, then
\begin{align}\label{Wdist}
d_W(F,N)\leq\frac{\sqrt{10}}{2\sigma^2}E\bigl[\HHnorm{\boldsymbol{D}^2F\otimes_1\boldsymbol{D}^2F}^2\bigr]^{\frac{1}{4}}\times E\bigl[\Hnorm{\boldsymbol{D}F}^4\bigr]^{\frac{1}{4}}
\end{align}
where $d_W$ is the Wasserstein distance given by $d_W(U,N)=\sup_{\{g:\left\|g\right\|_{\text{Lip}}\leq1\}}\abs{E[g(U)]-E[g(N)]}$. If, in addition, the law of $F$ is absolutely continuous with respect to the Lebesgue measure, then
\begin{align}\label{TVdist}
d_{TV}(F,N)\leq\frac{\sqrt{10}}{\sigma^2}E\bigl[\HHnorm{\boldsymbol{D}^2F\otimes_1\boldsymbol{D}^2F}^2\bigr]^{\frac{1}{4}}\times E\bigl[\Hnorm{\boldsymbol{D}F}^4\bigr]^{\frac{1}{4}}
\end{align}
where $d_{TV}$ is the total variation distance given by $d_{TV}(U,N)=\sup_{A\in\mathcal{B}(\RR)}\abs{P(U\in A)-P(N\in A)}$
\end{lemma}
\begin{preremark}\label{conditions}
This lemma is quite useful since it essentially tells us that in order to obtain a CLT result for a family of random variables $F_T$, we just need to check three conditions:
\begin{enumerate}

\item {\bf Expectation of the First Derivative's Norm:}
\begin{align}\label{cond1}
E\bigl[\Hnorm{\boldsymbol{D}F_T}^4\bigr] = O(1)\ \text{ as }\ T\to\infty
\end{align}

\item {\bf Expectation of the Contraction's Norm:}
\begin{align}\label{cond2}
E\bigl[\HHnorm{\boldsymbol{D}^2F_T\otimes_1\boldsymbol{D}^2F_T}^2\bigr]\rightarrow0\ \text{ as } \ T\rightarrow\infty
\end{align}

\item {\bf Existence of the Variance}
\begin{align}\label{cond3}
\text{Var}[F_T]\rightarrow \Sigma^2\in(0,\infty)\ \text{ exists as }\ T\to\infty
\end{align}

\end{enumerate}
Due to the Gaussian Poincar$\acute{\text{e}}$ inequality, Var$[F_T]\leq E\bigl[\Hnorm{\boldsymbol{D}F_T}^2\bigr]\leq\sqrt{E\bigl[\Hnorm{\boldsymbol{D}F_T}^4\bigr]}$, so the variance will go to 0 if the expectation of the first Malliavin derivative's norm goes to 0. This is why condition (\ref{cond1}) is necessary, and the convergence to zero of the Wasserstein distance relies on condition (\ref{cond2}).
\end{preremark}

\subsection{Malliavin calculus on Wiener-Poisson space}

Let $\mathcal{H}=L^2_\mu$. Assume there is a complete probability space $(\Omega,\mathcal{F},\mathbb{P})$ where $L_t$ is a cadlag, centered, L$\acute{\text{e}}$vy process: $L_t$ has stationary and independent increments, is continuous in probability and $L_0=0$, with $E[L_1^2]<\infty$. At the risk of causing some confusion, is denoted by $\mathcal{F}$ the filtration generated by $L_t$ completed with the null sets of the above filtration, and work on the space $(\Omega,\mathcal{F},\mathbb{P})$. Assume this process is represented by the triplet $(0,\sigma^2,\nu)$, where $\nu$ is the L$\acute{\text{e}}$vy measure s.t. $d\mu(t,x)=\sigma^2dt\delta_0(x)+x^2dtd\nu(x)(1-\delta_0(x))$ and $\int_\RR x^2d\nu(x)<\infty$. This process can be represented as
$$L_t=\sigma W_t + \dint xd\widetilde{N}(t,x)$$
where $W_t$ is a standard Brownian motion, $\sigma\geq0$ and $\widetilde{N}$ is the compensated jump measure. A fuller exposition on L$\acute{\text{e}}$vy processes can be found in \cite{Applebaum} and \cite{Sato}. This process is extended to a random measure $M$, which is used to construct (in an analogous way to the It$\hat{\text{o}}$ integral construction) an integral on the step functions, and then by linearity and continuity it is extended to $L^2\bigl(([0,T]\times\RR)^q,\mathcal{B}([0,T]\times\RR)^q,\mu^{\otimes q}\bigr)$ and denoted by $I_q$. This integral satisfy the following properties:
\begin{enumerate}

\item $I_q(f)=I_q(\tilde{f})$

\item $I_q(af+bg)=aI_q(f)+bI_q(g)$ \hspace{1cm} ($a,b\in\RR$)

\item $E[I_p(f)I_q(g)]=q!\int_{(\RR^+\times\RR)^q}\tilde{f}\tilde{g}d\mu^{\otimes q}1_{\{q=p\}}$

\end{enumerate}
These properties are stated in \cite{Josep} and their proof can be found in \cite{Ito}. We have a product formula in this framework which is similar to the one in remark \ref{prodwiener} but with extra terms coming from the poisson integration part. A product formula for the pure jump framework can be found in \cite{Peccati}. Before stating the formula, is needed to define a general version of the contraction. Let $\phi\in L^2_{\mu^{\otimes k_1}}$ and $\psi\in L^2_{\mu^{\otimes k_2}}$ be symmetric functions. Then the general contraction of order $r\leq\min\{k_1,k_2\}$ and $s\leq\min\{k_1,k_2\}-r$ is given by the integration of the first $r$ variables and the ``sharing'' of the following $s$ variables, i.e., $\phi\otimes_r^s\psi=\prod_{i=1}^sz_{2i}\Lqprod{\phi(\cdot,z,x)}{\psi(\cdot,z,y)}{\mu^{\otimes r}}$, where $z\in(\RR^2)^s$ and $(x,y)\in(\RR^2)^{k_1-r-s}\times(\RR^2)^{k_2-r-s}$. Now the product formula can be stated as follows. If $\abs{f}\otimes_r^s\abs{g}\in L^2_{\mu^{p+q-2r-s}}$ for $0\leq r\leq \min\{p,q\}$ and $0\leq s\leq \min\{p,q\}-r$, then
\begin{align}\label{prodint}
I_p(f)I_q(g)=\sum_{r=0}^{p\wedge q}\sum_{s=0}^{p\wedge q-r}r!s! {p \choose r}{q \choose r}{p-r \choose s}{q-r \choose s}I_{p+q-2r-s}(f\otimes_r^s g)
\end{align}
The proof of this product formula can be found in \cite{Lee}.
\begin{preremark}
In the general contraction formula $z_{2i}$ is the size of the jump and $z_{2i-1}$ is the time when that jump occurs $\bigl(z=(z_1,z_2,\dots,z_{2s-1},z_{2s})\bigr)$. If we only have the Wiener part, the factor $z_{2i}$ would be zero unless $s=0$, and we then obtain the contraction defined in the Wiener space. Similarly, when the terms where $s\neq0$ are zero, the formula (\ref{prodint}) reduces to that in remark \ref{prodwiener}
\end{preremark}
It has been also proved by It$\hat{\text{o}}$ \cite[Theorem 2]{Ito} that for all $F\in L^2(\Omega):=L^2(\Omega,\mathcal{F},\mathbb{P})$, we have,
\begin{align}\label{chaosrep1}
F=\sum_{q=0}^\infty I_q(f_q),\hspace{1cm} f_q\in L^2_{\mu^{\otimes q}}:=L^2\bigl(([0,T]\times\RR)^q,\mathcal{B}([0,T]\times\RR)^q,\mu^{\otimes q}\bigr)
\end{align}
and that this representation is unique if the $f_q$'s are symmetric function. From this chaotic representation we can define the annihilation operators and creation operators, the former will be the Malliavin derivative and the latter will be the Skorohod integral. In this way define Dom$\boldsymbol{D}$ as the set of functionals $F\in L^2(\Omega)$ represented as in (\ref{chaosrep1}) such that $\sum_{q=1}^\infty qq!\Lqnorm{f_q}{\mu^{\otimes q}}^2<\infty$. For $F\in$Dom$\boldsymbol{D}$ the Malliavin derivative of F is the stochastic process given by 
\begin{align}\label{malliavinderivative2}
\boldsymbol{D}_zF=\sum_{q=0}^\infty qI_{q-1}(f_q(z,\cdot)),\hspace{1cm} z\in[0,T]\times\RR,\ f_q \text{ symmetric}
\end{align}
If we define the inner product as $\Lqprod{f}{g}{\mu}=\int_{\RR^+\times\RR}f(z)g(z) d\mu(z)$, then Dom$\boldsymbol{D}$ is a Hilbert space with the inner product $\left\langle F,G\right\rangle=E[FG]+E[\Lqprod{\boldsymbol{D}_zF}{\boldsymbol{D}_zG}{\mu}]$. We can embed Dom$\boldsymbol{D}$ in two spaces Dom$\boldsymbol{D}^0$ and Dom$\boldsymbol{D}^J$. Dom$\boldsymbol{D}^0$ is defined as the set of all functionals $F\in L^2(\Omega)$ with representation given as in (\ref{chaosrep1}) such that $\sum_{q=1}^\infty qq!\int_0^T\Lqnorm{f_q\bigl((t,0),\cdot\bigr)}{\mu^{\otimes(q-1)}}^2\hspace{-.8cm}dt<\infty$, while Dom$\boldsymbol{D}^J$ is defined as the respective functionals satisfying $\sum_{q=1}^\infty qq!\int_{[0,T]\times\RR_0}\Lqnorm{f_q(z,\cdot)}{\mu^{\otimes(q-1)}}^2\hspace{-.8cm}d\mu(z)<\infty$; hence Dom$\boldsymbol{D}$=Dom$\boldsymbol{D}^0\cap$Dom$\boldsymbol{D}^J$. We can now rewrite (due to the independency of $W$ and $\widetilde{N}$) $\Omega$ as the cross product $\Omega_W\times\Omega_{\widetilde{N}}$.
\begin{itemize}

\item {\bf Derivative $\boldsymbol{D}_{t,0}$}\\
This derivative can be interpreted as the derivative with respect to the Brownian motion part. Using the isometry $L^2(\Omega)\simeq L^2(\Omega_W;L^2(\Omega_{\widetilde{N}}))$, we can define a Malliavin derivative as we did in the Wiener case but using the $L^2(\Omega_{\widetilde{N}})$-valued smooth random variables $S_{\widetilde{N}}$, that is, for the functionals of the form $F=\sum_{i=1}^nG_iH_i$, where $G_i\in \mathcal{S}$ and $H_i\in L^2(\Omega_{\widetilde{N}})$. Then, this derivative will be $\boldsymbol{D}^WF=\sum_{i=1}^n(\boldsymbol{D}^WG_i)H_i$ and this $\boldsymbol{D}^WG_i$ is the derivative defined in (\ref{malliavinderivative}). This definition is extended (see \cite{Josep}) to a subspace Dom$\boldsymbol{D}^W\subset$Dom$\boldsymbol{D}^0$ and for $F\in$Dom$\boldsymbol{D}^W$,
\begin{align}\label{derivative0}
\boldsymbol{D}_{t,0}F=\frac{1}{\sigma}\boldsymbol{D}_t^WF
\end{align}
Furthermore, we also have a chain rule result for functionals of the form $F=f(G,H)\in L^2(\Omega)$ with $G\in$Dom$\boldsymbol{D}^W$, $H\in L^2(\Omega_{\widetilde{N}})$ and $f(x,y)$ continously differentiable in the variable $x$ with bounded partial derivative. We have that $F\in$Dom$\boldsymbol{D}^0$ and
\begin{align}\label{chainrule0}
\boldsymbol{D}_{t,0}F=\frac{1}{\sigma}\frac{\partial f}{\partial x}(G,H)\boldsymbol{D}_t^WG
\end{align}
This is also true (as in the Wiener space case) for functions which are a.e. differentiable but with the restriction that $G$ has an absolutely continuous law.

\item {\bf Derivative $\boldsymbol{D}_z$ ($z\neq(t,0)$)}\\
This derivate has been shown, in \cite{Josep0}, to be a difference operator. The idea is to introduce a jump of size $x$ at moment $t$. Then, the Malliavin derivative with $z=(t,x)$ will be the translation operator given by
$$\boldsymbol{D}_zF=\Psi_{t,x}F=\frac{F(\omega_{t,x})-F(\omega)}{x}$$
for any $F\in$Dom$\boldsymbol{D}^J$ such that $E\bigl[\dint_{[0,T]\times\RR_0}(\Psi_zF)^2d\mu(z)\bigr]<\infty$. See \cite{Josep0} for a complete contruction on the canonical space in which this is developed, and \cite{Josep} for a quick explanation on how to introduce a jump at moment $t$ and the conditions on the $\omega$'s. We also have a chain rule for this Malliavin derivative but using just the difference instead of the derivative, that is, for $F=f(G,H)\in L^2(\Omega)$ with $G\in L^2(\Omega_{\widetilde{W}})$, $H\in$Dom$\boldsymbol{D}^J$ and $f(x,y)$ continous it holds that,
$$\boldsymbol{D}_zF=\frac{f\bigl(G,H(\omega_{t,x})\bigr)-f\big(G,H(\omega)\bigr)}{x}=\frac{f\bigl(G,x\boldsymbol{D}_zH+H(\omega)\bigr)-f\big(G,H(\omega)\bigr)}{x}$$
Notice that if $f$ is differentiable, then we can use the mean value theorem to obtain
\begin{align}\label{chainrule1}
\boldsymbol{D}_zF=\frac{\partial f}{\partial y}\big(G,H(\omega)+\theta_z x\boldsymbol{D}_zH\bigr)\boldsymbol{D}_zH
\end{align}
for some $\theta_z\in(0,1)$

\end{itemize}

In the same way, consider the chaotic decomposition $F(z)=\sum_{q=0}^\infty I_q\bigl(f_q(z,\cdot)\bigr)$, with $f\in L^2_{\mu^{\otimes q}}$ symmetric with respect to the last $n$ variables. If $\sum_{q=0}^\infty (q+1)!\Lqnorm{f_q}{\mu^{\otimes(q+1)}}^2\hspace{-.7cm}<\infty$ then we say that $F\in$Dom$\delta$. Now we can define the Skorohod integral of $F\in$Dom$\delta$ by
\begin{align}\label{skorohod2}
\delta(F)=\sum_{q=0}^\infty I_{q+1}(\widetilde{f}_q)\in L^2(\Omega)
\end{align}
This operator is the adjoint of the operator $\boldsymbol{D}_z$, so $E[\delta(F)G]=E\bigl[\Lqprod{F(z)}{\boldsymbol{D}_zG}{\mu}\bigr]$ for all $G\in$Dom$\boldsymbol{D}$. Denote by $\mathbb{L}^{1,2}$ the set of elements $F\in L^2_{\mu\otimes\mathbb{P}}\bigl([0,T]\times\RR\times\Omega\bigr)$ such that $\sum_{q=1}^\infty qq!\Lqnorm{f_q}{\mu^{\otimes q}}^2\hspace{-.4cm}<\infty$. For all $F\in\mathbb{L}^{1,2}\subset$Dom$\delta$ we have that $F(z)\in$Dom$\boldsymbol{D}$, $\forall \ z \ \mu-a.e.$ and that $\boldsymbol{D}_\cdot F(\cdot)\in L^2_{\mu^{\otimes2}\times\mathbb{P}}\bigl(([0,T]\times\RR)^2\times\Omega\bigr)$.

\bigskip

Finally, the definitions of the Ornstein-Uhlenbeck semigroup $T_t$ and its infinitesimal generator $L$ are the same as in the Wiener space case. Basically, all we need to define it is the Malliavin derivative and the Skorohod integral, that is, we can just define $L=-\delta\boldsymbol{D}$. With this definition we obtain that for $F\in L^2(\Omega)$ with chaotic representation (\ref{chaosrep1}), $LF=\sum_{q=1}^\infty -qI_q(f_q)$ and $T_tF=\sum_{q=0}^\infty e^{-qt}I_q(f_q)$. Samely, the pseudo-inverse is given by $L^{-1}F=\sum_{q=1}^\infty\frac{-1}{q}I_q(f_q)$ and $LL^{-1}F=F-E[F].$

\section{Main theorems}

The first tool needed is the extension of the so-called {\em Gaussian Poincar$\acute{\text{e}}$ inequality} for the present context. But to prove this is required to have an inequality similar to the one proved in \cite[Proposition 3.1]{Nourdin} (was proved for all $p\geq2$ in the Wiener space case). The technique used in their proof was based on the equivalence between Mehler and Ornstein-Uhlenbeck semigroups for the Gaussian case, but in the Wiener-Poisson space we lack such an equivalence. Nevertheless, it is possible to prove it for $p=2$ and that is, in fact, the one needed to prove the extension of the Gaussian Poincar$\acute{\text{e}}$ inequality.
\begin{prop}\label{ineq1}
Let $F\in$Dom$\boldsymbol{D}$ such that $E[F]=0$. Then,
$$E\bigl[\Lqnorm{\boldsymbol{D}L^{-1}F}{\mu}^2\bigr]\leq E\bigl[\Lqnorm{\boldsymbol{D}F}{\mu}^2\bigr]$$
\end{prop}
\begin{proof}
Assume $F$ has its chaos decomposition given by (\ref{chaosrep1}). By the orthogonality between chaoses we get,
\begin{align*}
\hspace{-1cm} E\bigl[\Lqnorm{\boldsymbol{D}L^{-1}F}{\mu}^2\bigr]=&E\biggl[\Lqnorm{\sum_{q=1}^\infty\boldsymbol{D}\frac{1}{q}I_q(f_q)}{\mu}^2\biggr]=E\biggl[\sum_{q=1}^\infty\frac{1}{q^2}\Lqnorm{\boldsymbol{D}I_q(f_q)}{\mu}^2\biggr]\leq E\biggl[\sum_{q=1}^\infty\Lqnorm{\boldsymbol{D}I_q(f_q)}{\mu}^2\biggr] =E\bigl[\Lqnorm{\boldsymbol{D}F}{\mu}^2\bigr]
\end{align*}
\end{proof}

\begin{theorem}\label{poincareineq}
(Extension of the Gaussian Poincar$\acute{\text{e}}$ inequality)\\
Let $F\in$Dom$\boldsymbol{D}$. Then,
\begin{align}\label{poincareineq2}
\text{Var}[F]\leq E\big[\Lqnorm{\boldsymbol{D}F}{\mu}^2\bigr]
\end{align}
with equality if and only if $F$ is a linear combination of elements in the first and $0^{\text{th}}$ chaos.
\end{theorem}
\begin{proof}
Assume, without loss of generality, that $E[F]=0$.
\begin{align*}
\text{Var}[F]=E[F^2]=E\bigl[\Lqprod{\boldsymbol{D}F}{\boldsymbol{D}L^{-1}F}{\mu}\bigr]\leq E\bigl[\Lqnorm{\boldsymbol{D}F}{\mu}^2\bigr]^{\frac{1}{2}}E\bigl[\Lqnorm{\boldsymbol{D}L^{-1}F}{\mu}^2\bigr]^{\frac{1}{2}}\leq E\bigl[\Lqnorm{\boldsymbol{D}F}{\mu}^2\bigr]
\end{align*}
where Proposition \ref{ineq1} was used in the last step, and the fact that $F=\delta\boldsymbol{D}L^{-1}F$ in the second step.
\end{proof}

Notice is possible to combine formulas (\ref{chainrule0}) and a version of (\ref{chainrule1}) to write the chain rule in a unique way when $f$ is a function of just one variable which is $k$ times continously differentiable. Similarly, is possible to give a unified formula for the derivative of a product. Using the fact that the jump $x$ is zero for $\boldsymbol{D}_{t,0}$, the following unified chain and product rules are obtained,
\begin{prop}\label{chainrule2th}(Chain rules and Product rule)\\
Let $F, G, H\in$Dom$\boldsymbol{D}^W\cap$Dom$\boldsymbol{D}^J$ such that $\boldsymbol{D}F, \boldsymbol{D}G, \boldsymbol{D}H \in L^2_{\mu}$. Also consider $f,g\in\mathcal{C}^{k-1}$, both with bounded first derivatives ($f'$ and $g'$ can be unbounded if the the law of $F$ is absolutely continuous with respect to the Lebesgue measure), such that $f$ is $k$-times differentiable and $g^{(k-1)}$ is a.e. differentiable. Then,
\begin{align}\label{chainrule2}
\boldsymbol{D}_zf(F) = \sum_{n=1}^{k-1}\frac{f^{(n)}(F)}{n!}x^{n-1}(\boldsymbol{D}_zF)^n+\frac{f^{(k)}(F+\theta_zx\boldsymbol{D}_zF)}{k!}x^{k-1}(\boldsymbol{D}_zF)^k
\end{align}
for some function $\theta_z\in(0,1)$ for all $z=(t,x)\in\RR^+\times\RR$,
\begin{align}\label{chainrule3}
\boldsymbol{D}_zg(F) = \sum_{n=1}^{k-1}\frac{g^{(n)}(F)}{n!}x^{n-1}(\boldsymbol{D}_zF)^n+\int_0^{\boldsymbol{D}_zF}\frac{g^{(k)}(F+xu)}{(k-1)!}x^{k-1}(\boldsymbol{D}_zF-u)^{k-1}du
\end{align}
and
\begin{align}\label{productrule2}
\boldsymbol{D}_z(GH)=\boldsymbol{D}_zG\cdot H+G\cdot\boldsymbol{D}_zH+x\cdot\boldsymbol{D}_zG\cdot\boldsymbol{D}_zH
\end{align}
for all $z=(t,x)\in\RR^+\times\RR$.
\end{prop}
\begin{proof}
If $z=(t,0)$, we get the chain rule formula (\ref{chainrule0}) and the usual product rule in Wiener space. If $z=(t,x)$ with $x\neq0$, then the Malliavin derivative formula tells us that $\boldsymbol{D}_zf(F)=\frac{f(F(\omega_{t,x}))-f(F(\omega))}{x}$. Since $f\in\mathcal{C}^{k-1}$ and $k$-times differentiable, $f(y)=f(y_0)+\sum_{n=1}^{k-1}\frac{f^{(n)}(y_0)}{n!}(y-y_0)^n+\frac{f^{(k)}(y_0+\theta_y(y-y_0))}{k!}(y-y_0)^k$ for some $\theta_y\in(0,1)$ (Taylor series of $f$ with mean-value form for the remainder). Using this expansion with $y=F(\omega_z)$, $y_0=F(\omega)$ and recalling that $y-y_0=F(\omega_z)-F(\omega)=x\boldsymbol{D}_zF$, the chain rule result (\ref{chainrule2}) will follows.

The second chain rule formula is obtained by using the Taylor expansion for $g$ with integral form for the remainder, i.e. $g(y)=g(y_0)+\sum_{n=1}^{k-1}\frac{g^{(n)}(y_0)}{n!}(y-y_0)^n+\int_{y_0}^y\frac{g^{(k)}(v)}{(k-1)!}(y-v)^{k-1}dv$. By using the values for $y$ and $y_0$ as in the previous case, and applying the change of variable $v=F+xu$ in the integral we get the chain rule formula (\ref{chainrule3}).

For the product rule we get trivially that

\bigskip

\noindent $\boldsymbol{D}_z(GH)=\frac{G(\omega_{t,x})H(\omega_{t,x})-G(\omega)H(\omega)}{x}=\frac{G(\omega_{t,x})-G(\omega)}{x}H(\omega)+G(\omega)\frac{H(\omega_{t,x})-H(\omega)}{x}+x\frac{G(\omega_{t,x})-G(\omega)}{x}\frac{H(\omega_{t,x})-H(\omega)}{x}$

\end{proof}

\subsection{Upper bound and second order Poincar$\acute{\text{e}}$ inequality}

The main theorem of this paper is now stated. Recall the following bound on the Wasserstein distance\footnote{see \cite{Nourdin0} for further details on this bound},
$$d_W(F,N)\leq \sup_{f\in\mathscr{F}_W}\abs{E[f'(F)-Ff(F)]}$$
where $N\sim\mathcal{N}(0,1)$ and $\mathscr{F}_W:=\{f\in\mathcal{C}^1\ /\ f' \text{ is Lipschitz},\ \Linfnorm{f'}\leq1, \Linfnorm{f''}\leq2\}$\footnote{ The class $\mathscr{F}_W$ is the collection of all continously differentiable functions $f:R\rightarrow R$ that has derivative bounded by 1 and such that there exists a version of $f''$ that is bounded by 2}.

\begin{theorem}\label{main}(Upper Bound)\\
Let $N\sim\mathcal{N}(0,1)$ and let $F\in$Dom$\boldsymbol{D}^W\cap$Dom$\boldsymbol{D}^J$ be such that $E[F]=0$. Then,
\begin{align}\label{mainineq1}
d_W(F,N)\leq E\abs{1-\Lqprod{\boldsymbol{D}F}{-\boldsymbol{D}L^{-1}F}{\mu}}+E\bigl[\Lqprod{\abs{x(\boldsymbol{D}F)^2}}{\abs{\boldsymbol{D}L^{-1}F}}{\mu}\bigr]
\end{align}
\end{theorem}
\begin{proof}
By Proposition \ref{chainrule2th} (with $k=2$ in (\ref{chainrule3})) we get $\boldsymbol{D}f(F)=f'(F)\boldsymbol{D}F+\int_0^{\boldsymbol{D}F}f''(F+xu)x(\boldsymbol{D}F-u)du$. On the other hand, using the identity $F=-\delta\boldsymbol{D}L^{-1}F$ (recall $E[F]=0$) and the integration by parts formula we get $E[Ff(F)]=E\bigl[\Lqprod{\boldsymbol{D}f(F)}{-\boldsymbol{D}L^{-1}F}{\mu}\bigr]$. Putting them together we get
\begin{align*}
\abs{E[f'(F)-Ff(F)]}&=\abs{E[f'(F)]-E\bigl[\Lqprod{\boldsymbol{D}f(F)}{-\boldsymbol{D}L^{-1}F}{\mu}\bigr]}\\
&=\abs{E[f'(F)]-E\biggl[f'(F)\Lqprod{\boldsymbol{D}F}{-\boldsymbol{D}L^{-1}F}{\mu}+\Lqprod{\int_0^{\boldsymbol{D}F}f''(F+xu)x(\boldsymbol{D}F-u)du}{-\boldsymbol{D}L^{-1}F}{\mu}\biggr]}\\
&\leq E\biggl[\abs{f'(F)}\abs{1-\Lqprod{\boldsymbol{D}F}{-\boldsymbol{D}L^{-1}F}{\mu}}\biggr]+E\biggl[\Lqprod{\int_0^{\boldsymbol{D}F}\abs{f''(F+xu)}\abs{x(\boldsymbol{D}F-u)}du}{\abs{\boldsymbol{D}L^{-1}F}}{\mu}\biggr]\\
&\leq \Linfnorm{f'}E\biggl[\abs{1-\Lqprod{\boldsymbol{D}F}{-\boldsymbol{D}L^{-1}F}{\mu}}\biggr]+E\biggl[\Lqprod{\Linfnorm{f''}\frac{\abs{x(\boldsymbol{D}_zF)^2}}{2}}{\abs{\boldsymbol{D}L^{-1}F}}{\mu}\biggr]
\end{align*}
Finally, use the fact that $\Linfnorm{f'}\leq1$ and $\Linfnorm{f''}\leq2$ to obtain the result.
\end{proof}
Is desirable to use this result to obtain a nice upper bound, as in the Wiener space with Lemma \ref{nourdinineq}. The main problem faced is that equivalence between the O-U and Mehler semigroups is no longer available, so the proofs of $E[\Lqnorm{\boldsymbol{D}^2L^{-1}F}{\mu}^p]\leq\frac{1}{2^p}E[\Lqnorm{\boldsymbol{D}^2F}{\mu}^p]$ and $E[\opnorm{\boldsymbol{D}^2L^{-1}F}^p]\leq\frac{1}{2^p}E[\opnorm{\boldsymbol{D}^2F}^p]$ given in the Wiener space case fail for the Wiener-Poisson space case. Nevertheless, is still possible to state an equivalent version of Lemma \ref{nourdinineq} for the case when $F$ lies in one specific chaos.
\begin{corollary}\label{mainqchaos}(Second order Poincar$\acute{\text{e}}$ inequality)\\
Fix $q\in\NN$ and let $F=I_q(f)$ with $E[F]=\mu$ and Var$[F]=\sigma^2$. Assume that $N\sim\mathcal{N}(\mu,\sigma^2)$, then
\begin{align}\label{ineqqchaos}
d_W(F,N)&\leq \frac{\sqrt{2}}{q\sigma^2}\biggl(2E\bigl[\opnorm{\boldsymbol{D}^2F}^4\bigr]^{\frac{1}{4}}E\bigl[\Lqnorm{\boldsymbol{D}F}{\mu}^4\bigr]^{\frac{1}{4}}+E\biggl[\Lqnorm{\Lqprod{x}{(\boldsymbol{D}^2F)^2}{\mu}}{\mu}^2\biggr]^{\frac{1}{2}}\biggr)+\frac{1}{q\sigma^3}E\bigl[\Lqprod{\abs{x}}{\abs{\boldsymbol{D}F}^3}{\mu}\bigr]\\
&\leq\frac{\sqrt{2}}{q\sigma^2}\biggl(2E\bigl[\Lqnorm{\boldsymbol{D}^2F\otimes_1\boldsymbol{D}^2F}{\mu^{\otimes2}}^2\bigr]^{\frac{1}{4}}E\bigl[\Lqnorm{\boldsymbol{D}F}{\mu}^4\bigr]^{\frac{1}{4}}+E\biggl[\Lqnorm{\Lqprod{x}{(\boldsymbol{D}^2F)^2}{\mu}}{\mu}^2\biggr]^{\frac{1}{2}}\biggr)+\frac{1}{q\sigma^3}E\bigl[\Lqprod{\abs{x}}{\abs{\boldsymbol{D}F}^3}{\mu}\bigr]
\end{align}
\end{corollary}
\begin{proof}
Assume without loss of generality that $\mu=0$ and $\sigma^2=1$. By Theorem \ref{main} and H$\ddot{\text{o}}$lder we have that
$$d_W(F,N)\leq E\bigl[\bigl(1-\Lqprod{\boldsymbol{D}F}{-\boldsymbol{D}L^{-1}F}{\mu}\bigr)^2\bigr]^{\frac{1}{2}}+E\bigl[\Lqprod{\abs{x(\boldsymbol{D}F)^2}}{\abs{\boldsymbol{D}L^{-1}F}}{\mu}\bigr]$$
Also, notice that $E[\Lqprod{\boldsymbol{D}F}{-\boldsymbol{D}L^{-1}F}{\mu}]=E[-\delta\boldsymbol{D}L^{-1}F\cdot F]=E[F^2]=1$, so if $G=\Lqprod{\boldsymbol{D}F}{-\boldsymbol{D}L^{-1}F}{\mu}$ then $E\bigl[\bigl(1-\Lqprod{\boldsymbol{D}F}{-\boldsymbol{D}L^{-1}F}{\mu}\bigr)^2\bigr]=$Var$[G]$. By Theorem \ref{poincareineq} we have that Var$[G]\leq E[\Lqnorm{\boldsymbol{D}G}{\mu}^2]$. Also, by the product rule (\ref{productrule2}) we have that 
$$\boldsymbol{D}G=\Lqprod{\boldsymbol{D}^2F}{-\boldsymbol{D}L^{-1}F}{\mu}+\Lqprod{\boldsymbol{D}F}{-\boldsymbol{D}^2L^{-1}F}{\mu}+\Lqprod{x\boldsymbol{D}^2F}{-\boldsymbol{D}^2L^{-1}F}{\mu}$$
Putting all together and using the fact that $-L^{-1}F=\frac{1}{q}F$ we get,
\begin{align*}
E\bigl[\bigl(1-\Lqprod{\boldsymbol{D}F}{-\boldsymbol{D}L^{-1}F}{\mu}\bigr)^2\bigr]^{\frac{1}{2}}\leq \frac{\sqrt{2}}{q}\biggl(2E\biggl[\Lqnorm{\Lqprod{\boldsymbol{D}^2F}{\boldsymbol{D}F}{\mu}}{\mu}^2\biggr]^{\frac{1}{2}}+E\biggl[\Lqnorm{\Lqprod{x}{(\boldsymbol{D}^2F)^2}{\mu}}{\mu}^2\biggr]^{\frac{1}{2}}\biggr)
\end{align*}
The first term in the right is bounded above in the following way
$$E\biggl[\Lqnorm{\Lqprod{\boldsymbol{D}^2F}{\boldsymbol{D}F}{\mu}}{\mu}^2\biggr]^{\frac{1}{2}}\leq E\bigl[\opnorm{\boldsymbol{D}^2F}^4\bigr]^{\frac{1}{4}}E\bigl[\Lqnorm{\boldsymbol{D}F}{\mu}^4\bigr]^{\frac{1}{4}}$$
since $\Lqnorm{\Lqprod{\boldsymbol{D}^2F}{\boldsymbol{D}F}{\mu}}{\mu}^2\leq \opnorm{\boldsymbol{D}^2F}^2\Lqnorm{\boldsymbol{D}F}{\mu}^2$ and by H$\ddot{\text{o}}$lder. On the other hand, and again using the fact that $F$ is in the $q^{\text{th}}$ chaos, we get
$$E\bigl[\Lqprod{\abs{x(\boldsymbol{D}F)^2}}{\abs{\boldsymbol{D}L^{-1}F}}{\mu}\bigr]= \frac{1}{q}E\bigl[\Lqprod{\abs{x}}{\abs{\boldsymbol{D}F}^3}{\mu}\bigr]$$
thus, obtaining the first inequality. For the second inequality, it suffices to see that if $\{\gamma_j\}_{j\geq1}$ is the sequence of random eigenvalues of the random Hilbert-Schmidt operator $f\rightarrow\Lqprod{f}{D^2F}{\mu^{\otimes2}}$, then
$$\opnorm{\boldsymbol{D}^2F}^4=\max_{j\geq1}\abs{\gamma_j}^4\leq\sum_{j\geq1}\abs{\gamma_j}^4=\Lqnorm{\boldsymbol{D}^2F\otimes_1\boldsymbol{D}^2F}{\mu^{\otimes2}}^2$$
\end{proof}

As in Lemma \ref{nourdinineq}, this corollary basically says that if we want to show a CLT for a family of random variables $F_T$ (living in a fixed chaos) it is sufficient to check the following conditions,
\begin{enumerate}

\item {\bf Expectation of the First Derivative's Norm:}
\begin{align}\label{2cond1}
E\bigl[\Lqnorm{\boldsymbol{D}F_T}{\mu}^4\bigr] = O(1)\ \text{ as }\ T\to\infty
\end{align}

\item {\bf Expectation of the Cube of the First Derivative's Norm:}
\begin{align}\label{2cond2}
E\bigl[\Lqprod{\abs{x}}{\abs{\boldsymbol{D}F_T}^3}{\mu}\bigr]\rightarrow0\ \text{ as } \ T\rightarrow\infty
\end{align}

\item {\bf Expectation of the Contraction's Norm:}
\begin{align}\label{2cond3}
E\biggl[\Lqnorm{\boldsymbol{D}^2F_T\otimes_1\boldsymbol{D}^2F_T}{\mu^{\otimes2}}^2\biggr]\rightarrow0\ \text{ as } \ T\rightarrow\infty
\end{align}

\item {\bf Expectation of the Squared Second Derivative's Norm:}
\begin{align}\label{2cond4}
E\biggl[\Lqnorm{\Lqprod{x}{(\boldsymbol{D}^2F_T)^2}{\mu}}{\mu}^2\biggr]\rightarrow0\ \text{ as } \ T\rightarrow\infty
\end{align}

\item {\bf Existence of the Variance}
\begin{align}\label{2cond5}
\text{Var}[F_T]\rightarrow \Sigma^2\in(0,\infty)\ \text{ exists as } T\to\infty
\end{align}

\end{enumerate}

\section{Special cases and applications}

\subsection{The Wiener space case:\\Linear functionals of Gaussian-subordinated fields}

When we are working in this space the jump size is always zero, so the upper bound for the Wasserstein distance becomes
\begin{align}\label{wienercase}
d_W(F,N)\leq E\abs{1-\Lqprod{\boldsymbol{D}F}{-\boldsymbol{D}L^{-1}F}{\mu}}
\end{align}
which coincides perfectly with the bound computed in \cite{Nourdin0}.
\begin{preremark}
It is important to stress that Theorem \ref{main} is not a direct extension of the inequality in \cite{Nourdin0}, even though they appear to be similar. This is because in the pure Wiener case, the Malliavin calculus theory is developed for more abstract Hilbert spaces than $L^2_{\mu}$, and the inequality proved therein,
$$d_W(F,N)\leq E\abs{1-\Hprod{\boldsymbol{D}F}{-\boldsymbol{D}L^{-1}F}}$$
holds for any Hilbert space $\mathcal{H}$. However, when $\mathcal{H}=L^2_{\mu}$, Theorem \ref{main} is indeed an extension of Theorem 3.1 in \cite{Nourdin0}.
\end{preremark}
The Wiener space is well understood and the upper bounds obtained from (\ref{wienercase}) are more powerful than Corollary \ref{mainqchaos}. In fact, Lemma \ref{nourdinineq} is true for all $F\in\mathbb{D}^{2,4}$ and not just for functionals in a fixed Wiener chaos. As an application of Lemma \ref{nourdinineq}, the authors of \cite{Nourdin} proved a very useful central limit theorem for linear functionals of Gaussian-subordinated fields. Before stating it, some notation is introduced: Let $X_t$ be a centered Gaussian stationary process and define $C(t)=E[X_0X_t]=E[X_sX_{t+s}]$, its covariance function. By remark \ref{hilbertderivative}, is known that the Malliavin derivative of $X_t$ is well defined. Let $T>0$, $Z\sim \mathcal{N}\bigl(0,C(0)\bigr)$ and $\ f:\RR\rightarrow \RR$ be a real function of class $\mathcal{C}^2$ not constant s.t. $E\bigl[|f(Z)|\bigr]<\infty$ and $E\bigl[|f''(Z)|^4\bigr]<\infty$. In order to simplify the notation, the following random sequence is defined,
$$F_T = T^{-\frac{1}{2}}\int_0^T\bigl(f(X_t)-E[f(Z)]\bigr)dt$$
The theorem is stated as follows,
\begin{lemma}\label{finitecase}
Suppose that $\int_\RR \abs{C(t)}<\infty$, and assume that $f$ is a {\bf symmetric} real function. Then $\lim_{T\rightarrow\infty}\text{Var}[F_T]:=\Sigma^2\in(0,\infty)$ exists and as $T\rightarrow\infty$
$$F_T\stackrel{\text{law}}{\longrightarrow}N\sim\mathcal{N}(0,\Sigma^2)$$
\end{lemma}
\begin{preremark}\label{quantrate}
One advantage of inequality (\ref{Wdist}) is the fact that it allows us to quantify rates of convergence to normality. Indeed, it has been proved (see \cite{Nourdin}) that if $\widetilde{Z}\sim\mathcal{N}(0,1)$, then (as $T\rightarrow\infty$), $$d_W\biggl(\frac{F_T}{\sqrt{\text{Var}[F_T]}},\widetilde{Z}\biggr)=O(T^{-\frac{1}{4}})$$
\end{preremark}
Our goal in this subsection is to extend this result to the case when $\int_\RR \abs{C(t)}=\infty$. This is achievable under some conditions on the decay rate of the covariance. In fact, it is very convenient that for this functional the conditions (\ref{cond1}), (\ref{cond2}) and (\ref{cond3}) reduce to just one condition on the covariance of the underlying stationary Gaussian process $X_t$. Let $V(T)$ be a strictly positive continuous function with $V(T)\rightarrow0$ as $T\to\infty$ such that either $TV(T)\rightarrow0$ or $V\in\mathcal{C}^1$ and $TV'(T)\rightarrow0$ as $T\rightarrow\infty$. The following is the condition on the covariance that replace the three conditions on remark \ref{conditions}.\\

\noindent{\bf Condition $\boldsymbol{\ast}$:}
Either $\int_\RR \abs{C(t)}<\infty$ or $V(T)$ (with the above characteristics) exists such that,
$$\frac{C(T)}{V(T)}\xrightarrow[T\rightarrow \infty]{}M\neq0$$
$V(T)$ represents the decay rate for the covariance function.
Consider the following function
\begin{align*}
\widetilde{V}(T)=\begin{cases}
T & \text{if } \ \int_0^\infty\abs{C(x)}dx<\infty\\
\int_0^T\int_0^yV(x)dxdy & \text{if } \ \int_0^\infty\abs{C(x)}dx=\infty
\end{cases}
\end{align*}
Let $\mathcal{M}_C:=\{f\in\mathcal{C}^2\ /\ f \text{ is symmetric if } \int_\RR \abs{C(t)}<\infty \text{ or } E[f(Z)Z]\neq0 \text{ if } \int_\RR \abs{C(t)}=\infty\}$ and rewrite the functional $F_T$ as follows,
$$F_T = \widetilde{V}(T)^{-\frac{1}{2}}\int_0^T\bigl(f(X_t)-E[f(Z)]\bigr)dt$$

\begin{theorem}\label{main2}
Suppose that condition $\ast$ is verified by $C(t)$ and that $f\in\mathcal{M}_C$. Then $\lim_{T\rightarrow\infty}\text{Var}[F_T]:=\Sigma^2\in(0,\infty)$ exists and as $T\rightarrow\infty$
$$F_T\stackrel{\text{law}}{\longrightarrow}N\sim\mathcal{N}(0,\Sigma^2)$$
Furthermore, if $\int_\RR \abs{C(t)}=\infty$, then $\Sigma^2=2M\bigl(E[f(Z)Z]\bigr)^2$.
\end{theorem}
Before tackling this theorem, it is necessary to verify some facts that would simplify the proof.
\begin{prop}\label{computations}
Suppose that $\int_\RR \abs{C(t)}=\infty$. Then as $T\rightarrow\infty$,
\begin{enumerate}

\item $\bigl(\int_0^TV(x)dx\bigr)^{-1}\int_0^T\abs{C(t)}dt=O(1)$

\item $\widetilde{V}(T)^{-1}\int_{[0,T]^2}\abs{C(t-s)}dsdt=O(1)$

\item
\begin{itemize}

\item If $TV(T)\rightarrow0$:\\
$\widetilde{V}(T)^{-2}T\bigl(\int_0^T\abs{C(t)}dt\bigr)^3=O\bigl(\max\{V(T),TV(T)^2\bigl(\int_0^TV(x)dx\bigr)^{-1}\}\bigr)$

\item If $TV(T)\nrightarrow0$ and $TV'(T)\rightarrow0$\\
$\widetilde{V}(T)^{-2}T\bigl(\int_0^T\abs{C(t)}dt\bigr)^3=O\bigl(\max\{V(T),TV'(T)\}\bigr)$

\end{itemize}

\item For fixed $q\geq1$:
\begin{align*}
\widetilde{V}(T)^{-1}\int_{[0,T]^2}C(t-s)^qdsdt\rightarrow 2M1_{\{q=1\}}=
\begin{cases}
2M & \text{iff } \ q=1\\
0 & \text{iff } \ q\neq 1
\end{cases}
\end{align*}

\end{enumerate}
\end{prop}

\begin{proof}
The proof just involves simple applications of L'H$\hat{\text{o}}$pital's rule (L).
\begin{enumerate}

\item $$\lim_{T\rightarrow\infty}\frac{\int_0^T\abs{C(t)}dt}{\int_0^TV(x)dx}\stackrel{\text{L}}{=}\lim_{T\rightarrow\infty}\frac{\abs{C(T)}}{V(T)}=\abs{M}$$

\item Notice first that $\int_{[0,T]^2}\abs{C(t-s)}dsdt=2\int_0^T\int_0^t\abs{C(x)}dxdt$ so
$$\lim_{T\rightarrow\infty}\frac{\int_{[0,T]^2}\abs{C(t-s)}dsdt}{\widetilde{V}(T)}=\lim_{T\rightarrow\infty}\frac{2\int_0^T\int_0^t\abs{C(x)}dxdt}{\int_0^T\int_0^yV(x)dxdy}\stackrel{\text{L}}{=}2\lim_{T\rightarrow\infty}\frac{\abs{C(T)}}{V(T)}=2\abs{M}$$

\item
\begin{itemize}

\item If $TV(T)\rightarrow0$:\\
$$ \lim_{T\rightarrow\infty}\frac{T\bigl(\int_0^T\abs{C(t)}dt\bigr)^3}{\widetilde{V}(T)^2}=\lim_{T\rightarrow\infty}\biggl(\underbrace{\frac{\int_0^T\abs{C(t)}dt}{\int_0^TV(x)dx}}_{=O(1)}\biggr)^3\lim_{T\rightarrow\infty}\frac{T\bigl(\int_0^TV(x)dx\bigr)^3}{\widetilde{V}(T)^2}\stackrel{\text{L}}{=}O(1)\lim_{T\rightarrow\infty}\biggl(V(T)+\frac{3TV(T)^2}{2\int_0^TV(x)dx}\biggr)$$

\item If $TV(T)\nrightarrow0$ and $TV'(T)\rightarrow0$\\
$$ \lim_{T\rightarrow\infty}\frac{T\bigl(\int_0^T\abs{C(t)}dt\bigr)^3}{\widetilde{V}(T)^2}=\lim_{T\rightarrow\infty}\biggl(\underbrace{\frac{\int_0^T\abs{C(t)}dt}{\int_0^TV(x)dx}}_{=O(1)}\biggr)^3\lim_{T\rightarrow\infty}\frac{T\bigl(\int_0^TV(x)dx\bigr)^3}{\widetilde{V}(T)^2}\stackrel{\text{L}}{=}O(1)\lim_{T\rightarrow\infty}\bigl(4V(T)+3TV'(T)\bigr)$$

\end{itemize}

\item If for $q>1$, either $\lim_{T\rightarrow\infty}\int_{[0,T]^2}C(t-s)^qdsdt<\infty \ \text{ or }\ \lim_{T\rightarrow\infty}\int_0^TC(x)^qdx<\infty$, then the result will follow trivially. So let's assume that both go to infinity as $T$ goes to infinity.
$$ \lim_{T\rightarrow\infty}\frac{\int_{[0,T]^2}C(t-s)^qdsdt}{\widetilde{V}(T)}=\lim_{T\rightarrow\infty}\frac{2\int_0^T\int_0^tC(x)^qdxdt}{\int_0^T\int_0^yV(x)dxdy}\stackrel{\text{L}}{=}2\lim_{T\rightarrow\infty}\frac{C(T)^q}{V(T)}=2\lim_{T\rightarrow\infty}\underbrace{\biggl(\frac{C(T)}{V(T)}\biggr)^q}
_{\rightarrow M^q}\underbrace{V(T)^{(q-1)}}_{\rightarrow 0 \text{ if } q>1}=2M1_{\{q=1\}}$$

\end{enumerate}
\end{proof}
Now, to the proof of Theorem \ref{main2}
\begin{proof}
Notice that if $\int_\RR\abs{C(t)}<\infty$ then Theorem \ref{main2} reduces to Lemma \ref{finitecase} and there is nothing left to prove. Assume then, that $\int_\RR\abs{C(t)}=\infty$. Due to remark \ref{conditions}, it is enough to check that condition $\ast$ implies conditions (\ref{cond1}), (\ref{cond2}) and (\ref{cond3}).
\begin{itemize}

\item {\bf Expectation of the First Derivative's Norm:}\\
\begin{itemize}

\item First Malliavin Derivative:
$$\boldsymbol{D} F_T = \widetilde{V}(T)^{-\frac{1}{2}}\int_0^Tf'(X_t)1_{[0,t]}dt$$

\item Norm of the First Malliavin Derivative:
$$\Hnorm{\boldsymbol{D} F_T}^2 = \widetilde{V}(T)^{-1}\int_{[0,T]^2}\hspace{-.7cm}f'(X_t)f'(X_s)\Hprod{1_{[0,t]}}{1_{[0,s]}} dtds = \widetilde{V}(T)^{-1}\int_{[0,T]^2}\hspace{-.7cm}f'(X_t)f'(X_s)C(t-s) dtds$$
Then,
$$\Hnorm{\boldsymbol{D} F_T}^4 = \widetilde{V}(T)^{-2}\int_{[0,T]^4}\hspace{-.7cm}f'(X_t)f'(X_s)f'(X_u)f'(X_v)C(t-s)C(u-v) dtdsdudv$$

\item Expectation of the First Malliavin Derivative's Norm:\\
By using H$\ddot{\text{o}}$lder (twice) on the expectation and by the stationarity of $X_t$ we have the bound,
$$\abs{E[f'(X_t)f'(X_s)f'(X_u)f'(X_v)]}\leq E\bigl[\abs{f'(Z)}^4\bigr]$$
finally recovering the power we get,
$$E\bigl[\Hnorm{\boldsymbol{D} F_T}^4\bigr]\leq E[|f'(Z)|^4]\underbrace{\biggl(\widetilde{V}(T)^{-1}\int_{[0,T]^2}\abs{C(t-s)} dtds\biggr)^2}_{=O(1) \text{ by Proposition \ref{computations}}}$$

\end{itemize}
All this proves that,
\begin{align*}
\boxed{E\bigl[\Hnorm{\boldsymbol{D} F_T}^4\bigr]^{\frac{1}{4}}=O(1)\ \text{ as }\ T\to\infty}
\end{align*}

\item {\bf Expectation of the Contraction's Norm:}\\
In the same way we get,

\begin{itemize}

\item Second Malliavin Derivative:
$$\boldsymbol{D}^2 F_T = \widetilde{V}(T)^{-\frac{1}{2}}\int_0^Tf''(X_t)1_{[0,t]}^{\otimes2}dt$$

\item Contraction of Order 1:
\begin{align*}
\boldsymbol{D}^2F_T\otimes_1\boldsymbol{D}^2F_T =& \widetilde{V}(T)^{-1}\int_{[0,T]^2}\hspace{-.7cm}f''(X_t)f''(X_s)1_{[0,t]}\otimes 1_{[0,s]}\left\langle 1_{[0,t]},1_{[0,s]}\right\rangle_{\mathcal{H}} dtds\\
=& \widetilde{V}(T)^{-1}\int_{[0,T]^2}\hspace{-.7cm}f''(X_t)f''(X_s)1_{[0,t]}\otimes 1_{[0,s]}C(t-s) dtds
\end{align*}

\item Norm of the Contraction:
\begin{align*}
\HHnorm{\boldsymbol{D}^2F_T\otimes_1\boldsymbol{D}^2F_T}^2 &= \widetilde{V}(T)^{-2}\int_{[0,T]^4}\hspace{-.7cm}f''(X_t)f''(X_s)f''(X_u)f''(X_v)C(t-s)C(u-v)\\
&\hspace{5cm}\times\left\langle 1_{[0,t]},1_{[0,u]}\right\rangle_{\mathcal{H}}\left\langle 1_{[0,s]},1_{[0,v]}\right\rangle_{\mathcal{H}} dtdsdudv\\
&= \widetilde{V}(T)^{-2}\int_{[0,T]^4}\hspace{-.7cm}f''(X_t)f''(X_s)f''(X_u)f''(X_v)C(t-s)C(u-v)C(t-u)C(s-v) dtdsdudv
\end{align*}

\item Expectation of the Contraction's Norm:\\
By using H$\ddot{\text{o}}$lder in the same way as above we get,
$$E\bigl[\HHnorm{\boldsymbol{D}^2F_T\otimes_1\boldsymbol{D}^2F_T}^2\bigr] \leq E[|f''(Z)|^4]\widetilde{V}(T)^{-2}\int_{[0,T]^4}\abs{C(t-s)C(u-v)C(t-u)C(s-v)} dtdsdudv$$
Now, let's make the changes of variable $y=(t-s,u-v,t-u,v)$, and let's denote the new region by $\widetilde{\Omega}\times[0,T]$. So,
\begin{align*}
E\bigl[\HHnorm{\boldsymbol{D}^2F_T\otimes_1\boldsymbol{D}^2F_T}^2\bigr] &\leq E[|f''(Z)|^4]\widetilde{V}(T)^{-2}\int_0^T\int_{\widetilde{\Omega}}\abs{C(y_1)C(y_2)C(y_3)C(y_2+y_3-y_1)} dy\\
&= E[|f''(Z)|^4]\widetilde{V}(T)^{-2}T\int_{\widetilde{\Omega}}\abs{C(y_1)C(y_2)C(y_3)C(y_2+y_3-y_1)} dy_1dy_2dy_3
\end{align*}
Taking in account that by Cauchy-Schwarz, $\forall \ t\in\RR$,
$$\abs{C(t)}=\overbrace{\frac{\abs{E[X_0X_t]}}{\sqrt{\text{Var}[X_0]\text{Var}[X_t]}}}^{\leq 1}\overbrace{\sqrt{\text{Var}[X_0]\text{Var}[X_t]}}^{=C(0)}\leq C(0)$$
Also, it is clear that $\widetilde{\Omega}\subset[-T,T]^3$ and since the integrand is a non-negative even function, we can deduce that,
\begin{align*}
E\bigl[\HHnorm{\boldsymbol{D}^2F_T\otimes_1\boldsymbol{D}^2F_T}^2\bigr] &\leq E[|f''(Z)|^4]C(0)\widetilde{V}(T)^{-2}T\biggl(2\int_{[0,T]}\abs{C(y)}dy\biggr)^3\\
&= 8E[|f''(Z)|^4]C(0)\underbrace{\widetilde{V}(T)^{-2}T\biggl(\int_0^T\abs{C(y)}dy\biggr)^3}_{\xrightarrow[T\to\infty]{}0 \text{ by Proposition \ref{computations}}}
\end{align*}

All this proves that,
\begin{itemize}

\item If $TV(T)\rightarrow0$:\\
\begin{align*}
\hspace{-2cm} \boxed{E\bigl[\HHnorm{\boldsymbol{D}^2F_T\otimes_1\boldsymbol{D}^2F_T}^2\bigr]^{\frac{1}{4}}=O\biggl(\max\left\{V(T),TV(T)^2\biggl(\int_0^TV(x)dx\biggr)^{-1}\right\}^{\frac{1}{4}}\biggr)\ \text{ as }\ T\to\infty}
\end{align*}

\item If $TV(T)\nrightarrow0$ and $TV'(T)\rightarrow0$\\
\begin{align*}
\boxed{E\bigl[\HHnorm{\boldsymbol{D}^2F_T\otimes_1\boldsymbol{D}^2F_T}^2\bigr]^{\frac{1}{4}}=O\bigl(\max\{V(T),TV'(T)\}^{\frac{1}{4}}\bigr)\ \text{ as }\ T\to\infty}
\end{align*}

\end{itemize}
\end{itemize}

\item {\bf Existence of the Variance:}\\
Since $f\in\mathcal{M}_C$ then $E[f(X_0)X_0]=E[f(Z)Z]\neq0$. Also $H_1(x)=x$, so the first Hermite constant in the expansion (\ref{decomp}) is not 0, i.e., $c_1=E[f(X_0)X_0]\neq0$. Using the formula (\ref{covf}) for the covariance of $f$ we get,
\begin{align*}
\text{Var}[F_T]&=E\biggl[\biggl(\widetilde{V}(T)^{-\frac{1}{2}}\int_0^T\bigl(f(X_t)-E[f(Z)]\bigr)dt\biggr)^2\biggr]=\widetilde{V}(T)^{-1}\int_{[0,T]^2}\text{Cov}\bigl[f(X_t)f(X_s)\bigr]dtds\\
&=\widetilde{V}(T)^{-1}\int_{[0,T]^2}\sum_{q=1}^\infty c_q^2q!\bigl(E[X_tX_s]\bigr)^qdtds=\sum_{q=1}^\infty c_q^2q!\widetilde{V}(T)^{-1}\int_{[0,T]^2}C(t-s)^qdtds\\
&=c_1^2\underbrace{\frac{\int_{[0,T]^2}C(t-s)dsdt}{\widetilde{V}(T)}}_{\rightarrow 2M \text{ by Proposition \ref{computations}}}+\sum_{q=2}^\infty c_q^2q!\underbrace{\frac{\int_{[0,T]^2}C(t-s)^qdsdt}{\widetilde{V}(T)}}_{\rightarrow 0,\ \forall \ q \text{ by Proposition \ref{computations}}}\xrightarrow[T\rightarrow\infty]{}2Mc_1^2
\end{align*}
All this proves that,
\begin{align*}
\boxed{\lim_{T\rightarrow\infty}\text{Var}[F_T]=2M\bigl(E[f(Z)Z]\bigr)^2\in(0,\infty)\ \text{ exists}}
\end{align*}

\end{itemize}
Since the conditions were satisfied, Theorem \ref{main2} is proved.
\end{proof}

\begin{preremark}\label{quantrate2}\rm
Notice that during the proof of this theorem was possible to establish an estimate for the convergence rate to normality like in remark \ref{quantrate}, i.e., if $\widetilde{Z}\sim\mathcal{N}(0,1)$ then as $T\to\infty$,
\begin{itemize}

\item If $TV(T)\rightarrow0$:
$$d_W\biggl(\frac{F_T}{\sqrt{\text{Var}[F_T]}},\widetilde{Z}\biggr)=O\biggl(\max\left\{V(T),TV(T)^2\biggl(\int_0^TV(x)dx\biggr)^{-1}\right\}^{\frac{1}{4}}\biggr)$$

\item If $TV(T)\nrightarrow0$ and $TV'(T)\rightarrow0$:
$$d_W\biggl(\frac{F_T}{\sqrt{\text{Var}[F_T]}},\widetilde{Z}\biggr)=O\bigl(\max\{V(T),TV'(T)\}^{\frac{1}{4}}\bigr)$$

\end{itemize}
\end{preremark}

\subsubsection{Examples}
According to Theorem \ref{main2} the only condition we need to check in order to apply the central limit theorem to $F_T$ is the decay rate of the covariance function for the underlying stationary Gaussian process $X_t$ (condition $\ast$). In fact, if the decay rate is $t^{-\alpha}$ then we can apply the CLT if $\alpha\in(0,1)\cup(1,2)$, because in the case $\alpha\in(1,2)$ the integral $\int_\RR C(t)$ is finite and in the case $\alpha\in(0,1)$ the same integral is infinite but $V(T)=T^{-\alpha}\in\mathcal{C}^1$ and $TV'(T)=-\alpha T^{-\alpha}\rightarrow0$ as $T\rightarrow\infty$.

\begin{enumerate}

\item {\bf Fractional Brownian Motion (FBM):}\\
In this particular case is known that the difference of FBM is a stationary process for all $H\in(0,1)$. So, if $B^H_t$ is FBM then $X_t=B^H_{t+1}-B^H_t$ is a centered Gaussian stationary process. Its covariance function is,
$$C_1(T)=E[X_TX_0]=E[(B^H_{T+1}-B^H_T)(B^H_1-B^H_0)]=\frac{\abs{T+1}^{2H}+\abs{T-1}^{2H}-2T^{2H}}{2}$$
Thus,
$$\lim_{T\rightarrow\infty}T^{2-2H}C_1(T)=\lim_{T\rightarrow\infty}T^2\frac{\bigl(1+\frac{1}{T}\bigr)^{2H}+\bigl(1-\frac{1}{T}\bigr)^{2H}-2}{2}=H(2H-1)=M\in(0,\infty)$$
Then, the decay rate of its covariance function is $t^{2H-2}$, so Theorem \ref{main} is applicable to the increments of FBM for all $H\in\bigl(0,\frac{1}{2}\bigr)\cup\bigl(\frac{1}{2},1\bigr)$, and $F_T\stackrel{\text{law}}{\longrightarrow}N\sim\mathcal{N}(0,\Sigma^2)$ as $T\rightarrow\infty$

\item {\bf Ornstein-Uhlenbeck Driven by FBM:}\\
This process is given as the solution of the following SDE:
\begin{align}\label{SDE1}
Y^H_t=Y^H_0-\lambda\int_0^tY^H_sds+\widetilde{\sigma} B^H_t
\end{align}
where $\widetilde{\sigma},\lambda>0$ are constants, and $B_t^H$ is a fractional Brownian motion with Hurst parameter $H\in\bigl(0,\frac{1}{2}\bigr)\cup\bigl(\frac{1}{2},1\bigr)$. This process is stationary due to the stationarity of the increments of the FBM (used in the first example). So $X_t=Y^H_t-E[Y_0^H]$ is a centered Gaussian stationary process.
In \cite{Cheridito}, the authors proved the following lemma,
\begin{lemma}\label{lema1}
Let $H\in\bigl(0,\frac{1}{2}\bigr)\cup\bigl(\frac{1}{2},1\bigr)$ and $N\in\NN$. Then as $T\rightarrow\infty$,
$$C_2(T)=E[X_TX_0]=\text{Cov}[Y^H_TY^H_0]=\frac{\widetilde{\sigma}^2}{2}\sum_{n=1}^N\lambda^{-2n}\biggl(\prod_{k=0}^{2n-1}(2H-k)\biggr)T^{2H-2n}+O(T^{2H-2N-2})$$
\end{lemma}
which basically tells us that for all $H\in\bigl(0,\frac{1}{2}\bigr)\cup\bigl(\frac{1}{2},1\bigr)$ the decay rate of $C_2(T)$ is very similar to the decay rate of $C_1(T)$ (the covariance of the FBM increments). Lemma \ref{lema1} implies that,
$$\lim_{T\rightarrow\infty}T^{2-2H}C_2(T)=\frac{H(2H-1)\widetilde{\sigma}^2}{\lambda^2}=M\in(0,\infty)$$
As in example 1, due to this rate of decreasing, Theorem \ref{main} is applicable to this process for all $H\in\bigl(0,\frac{1}{2}\bigr)\cup\bigl(\frac{1}{2},1\bigr)$, and $F_T\stackrel{\text{law}}{\longrightarrow}N\sim\mathcal{N}(0,\Sigma^2)$ as $T\rightarrow\infty$

\end{enumerate}
According to remarks \ref{quantrate} and \ref{quantrate2} we can tell that for the above examples $F_T$ has a rate of convergence to normality of at least $T^{(\frac{1\vee (2H)}{4}-\frac{1}{2})}$ for all $H\in\bigl(0,\frac{1}{2}\bigr)\cup\bigl(\frac{1}{2},1\bigr)$, that is, for $\widetilde{Z}\sim\mathcal{N}(0,1)$,
$$d_W\biggl(\frac{F_T}{\sqrt{\text{Var}[F_T]}},\widetilde{Z}\biggr)=O\bigl(T^{(\frac{1\vee (2H)}{4}-\frac{1}{2})}\bigr)\ \text{ as }\ T\to\infty$$

\subsection{The Poisson space case:\\Simulation of small jumps}

Notice that in this space the measure $\mu$ has support contained in $\RR^+\times\RR_0$, and the formula obtained from the general case will be
\begin{align}\label{poissoncase}
d_W(F,N)\leq E\abs{1-\Lqprod{\boldsymbol{D}F}{-\boldsymbol{D}L^{-1}F}{\mu}}+E\bigl[\Lqprod{\abs{x(\boldsymbol{D}F)^2}}{\abs{\boldsymbol{D}L^{-1}F}}{\mu}\bigr]
\end{align}
\begin{preremark}\label{equivalency}
This particular case was worked out by G. Peccati, J. L. Sol$\acute{\text{e}}$, M. S. Taqqu and F. Utzet in \cite{Peccati}, with several examples (and conditions) given. They obtained the following inequality,
\begin{align}\label{poissoncase2}
d_W(F,N)\leq E\abs{1-\Lqprod{\widetilde{\boldsymbol{D}}F}{-\widetilde{\boldsymbol{D}}L^{-1}F}{\widetilde{\mu}}}+E\bigl[\Lqprod{\abs{(\widetilde{\boldsymbol{D}}F)^2}}{\abs{\widetilde{\boldsymbol{D}}L^{-1}F}}{\widetilde{\mu}}\bigr]
\end{align}
where the factor $x$ of the second term is missing. It is important to stress that both formulas are equivalent and that the difference lies in the definition of the Malliavin derivative and the random measure. In fact, the definition of Malliavin derivative used by them is $\widetilde{\boldsymbol{D}}_zF=F(\omega_z)-F(\omega)$ and the random measure is $\int_Ad\widetilde{N}(t,x)$ (instead of $\int_Axd\widetilde{N}(t,x)$ as in our case). So, the kernels of both chaos decomposition will differ by a factor of $x$, that is, if $F$ has chaos representation as in (\ref{chaosrep1}) for our framework with kernels $f_q$, then
$$F=\sum_{q=0}^\infty \widetilde{I}_q(xf_q)$$
in their framework ($\widetilde{I}_q$ is the corresponding extension with their random measure). For example, the Ornstein-Uhlenbeck process (with inital value 0) is given by
$$Y_t=\int_0^te^{-\lambda(t-s)}xd\widetilde{N}(s,x)=I_1\bigl(e^{-\lambda(t-s)}\bigr)=\widetilde{I}_1\bigl(xe^{-\lambda(t-s)}\bigr)$$
and it is easy to check that $\widetilde{\boldsymbol{D}}Y_t=xe^{-\lambda(t-s)}=x\boldsymbol{D}Y_t$. Also, since $d\mu(t,x)=x^2d\widetilde{\mu}(t,x)$, we have that
$$\Lqprod{\widetilde{\boldsymbol{D}}F}{-\widetilde{\boldsymbol{D}}L^{-1}F}{\widetilde{\mu}}=\Lqprod{\boldsymbol{D}F}{-\boldsymbol{D}L^{-1}F}{\mu}$$
and
$$\Lqprod{\abs{(\widetilde{\boldsymbol{D}}F)^2}}{\abs{\widetilde{\boldsymbol{D}}L^{-1}F}}{\widetilde{\mu}}=\Lqprod{\abs{x(\boldsymbol{D}F)^2}}{\abs{\boldsymbol{D}L^{-1}F}}{\mu}$$
\end{preremark}

\begin{preremark}\label{multidimensional}
In \cite{Peccati2} the authors accomplish the remarkable generalization of this bound to the multi-dimensional case. Their theorem reads as follows,
\begin{lemma}\label{multidimtheo}
Fix $d\geq2$ and let $C=\{C(i,j)\}_{0\leq i,j\leq d}$ be a $d\times d$ positive definite matrix. Suppose that $N\sim\mathcal{N}_d(0,C)$ and that $F=(F_1,\dots,F_d)$ is a $\RR^d$-valued random vector such that $E[F_i]=0$ and $F_i\in$Dom$\widetilde{\boldsymbol{D}}$ for all $i=1,\dots,d$. Then,\footnote{see \cite{Peccati2} for the definitions}
\begin{align*}
d_2(F,N)\leq& \opnorm{C^{-1}}\opnorm{C}^{\frac{1}{2}}E\left[\sum_{i,j}^d\bigl(C(i,j)-\Lqprod{\widetilde{\boldsymbol{D}}F_i}{-\widetilde{\boldsymbol{D}}L^{-1}F_j}{\widetilde{\mu}}\bigr)^2\right]^{\frac{1}{2}}\\
&+\frac{\sqrt{2\pi}}{8}\opnorm{C^{-1}}^{\frac{3}{2}}\opnorm{C}E\left[\Lqprod{\left(\sum_{i=1}^d\abs{\widetilde{\boldsymbol{D}}F_i}\right)^2}{\sum_{i=1}^d\abs{\widetilde{\boldsymbol{D}}L^{-1}F_i}}{\widetilde{\mu}}\right]
\end{align*}
\end{lemma}
From this inequality they conclude a useful result for the first chaos case.
\begin{lemma}\label{multdimcond}
For a fixed $d\geq2$, let $N\sim\mathcal{N}_d(0,C)$, with $C$ positive definite, and let
$$F_n=(F_{n,1},\dots,F_{n,d})=\bigl(\widetilde{I}_1\bigl(\widetilde{h}_{n,1}\bigr),\dots,\widetilde{I}_1\bigl(\widetilde{h}_{n,d}\bigr)\bigr),\ n\geq1$$
Let $K_n$ be the covariance matrix of $F_n$, that is, $K_n(i,j)=E\bigl[\widetilde{I}_1\bigl(\widetilde{h}_{n,i}\bigr)\widetilde{I}_1\bigl(\widetilde{h}_{n,j}\bigr)\bigr]=\Lqprod{\widetilde{h}_{n,i}}{\widetilde{h}_{n,j}}{\widetilde{\mu}}$. Then $F_n\stackrel{law}{\longrightarrow}N$ if $K_n(i,j)\xrightarrow[n\to\infty]{} C(i,j)$ and $\Lqnorm{\abs{\widetilde{h}_{n,i}}^{\frac{3}{2}}}{\widetilde{\mu}}^2\xrightarrow[n\to\infty]{}0$ for all $i,j=1\dots d$.
\end{lemma}

By remark \ref{equivalency}, we know that it is possible to rewrite these conditions in terms of ``our'' framework. Indeed, denoting $\widetilde{h}_{n,i}=xh_{n,i}$, it follows that $K_n(i,j)=\Lqprod{\widetilde{h}_{n,i}}{\widetilde{h}_{n,j}}{\widetilde{\mu}}=\Lqprod{h_{n,i}}{h_{n,j}}{\mu}$ and $\Lqnorm{\abs{\widetilde{h}_{n,i}}^{\frac{3}{2}}}{\widetilde{\mu}}^2=\Lqnorm{\abs{x}^{\frac{1}{2}}\abs{h_{n,i}}^{\frac{3}{2}}}{\mu}^2$.
\end{preremark}

In \cite{Asmussen}, the authors proved that the small jumps from a L$\acute{\text{e}}$vy process can be approximated by Brownian motion. Before this theorem is stated, some notation needs to be introduced: Let $X_t$ be a L$\acute{\text{e}}$vy process with triplet $(b,\sigma^2,\nu)$. To isolate the small jumps, consider the variance $\sigma(\epsilon)^2=\int_{\{\abs{x}\leq \epsilon\}}x^2d\nu(x)$ and the small jumps process $X_t^{\epsilon}=\sigma(\epsilon)^{-1}\int\hspace{-.2cm}\int_{[0,t]\times\{\abs{x}\leq \epsilon\}}xd\widetilde{N}(s,x)$. So $X_t=b_{\epsilon}t+\sigma W_t+N_t^{\epsilon}+\sigma(\epsilon)X_t^{\epsilon}$ where $N_t^{\epsilon}=t\sum_{s<t}\Delta X_s1_{\{\abs{\Delta X_s}\geq\epsilon\}}$ is the part of (finitely many) jumps bigger than $\epsilon$. Their theorem reads as follows,
\begin{lemma}\label{levycase}
$X_t^{\epsilon}\stackrel{law}{\longrightarrow}\widetilde{W}_t$ (Brownian motion independent of $W_t$ and $N_t^{\epsilon}$ for all $\epsilon$) as $\epsilon\rightarrow0$ if and only if for each $\kappa>0$, $\sigma\bigl(\kappa\sigma(\epsilon)\wedge\epsilon\bigr)\sim\sigma(\epsilon)$ as $\epsilon\rightarrow0$
\end{lemma}
They also proved that the condition $\sigma\bigl(\kappa\sigma(\epsilon)\wedge\epsilon\bigr)\sim\sigma(\epsilon)$ as $\epsilon\rightarrow0$ is implied by $\lim_{\epsilon\rightarrow0}\frac{\sigma_t(\epsilon)}{\epsilon}=\infty$. The importance of this lemma is that $X_t\stackrel{\text{law}}{\approx}b_{\epsilon} t+\sqrt{\sigma^2+\sigma(\epsilon)^2}W_t+N_t^{\epsilon}$ (for $\epsilon$ small enough), and the latter is quite easy to simulate.

\bigskip

The objective of this subsection is to extend this kind of result to functionals that are not necessarily L$\acute{\text{e}}$vy. To focus just on the jump part, let's assume, without loss of generality, that the triplet of the L$\acute{\text{e}}$vy process $X_t$ is $(0,0,\nu)$. Then, it can be written as $X_t=I_1(1_{[0,t]})=\int\hspace{-.2cm}\int_{[0,t]\times\RR}xd\widetilde{N}(s,x)$. Define $\widetilde{\sigma}_t(\epsilon)^2=\Lqnorm{h_t1_{[0,t]\times\{\abs{x}<\epsilon\}}}{\mu}=\int\hspace{-.2cm}\int_{[0,t]\times\{\abs{x}<\epsilon\}}h_t(s,x)^2x^2d\nu(x)ds$ for some $h_t\in L^2_{\mu}$ and consider the processes $\widetilde{X}_t=I_1\bigl(h_t1_{[0,t]}\bigr)$, $\widetilde{X}_t^{\epsilon}=I_1\bigl(\widetilde{\sigma}_t(\epsilon)^{-1}h_t1_{[0,t]\times\{\abs{x}<\epsilon\}}\bigr)$. Similarly, define $\widehat{\sigma}(\epsilon)^2=\int_{\{\abs{x}<\epsilon\}}x^2d\nu(x)$ and $\widehat{X}_t^{\epsilon}=I_1\bigl(\widehat{\sigma}(\epsilon)^{-1}h_t1_{[0,t]\times\{\abs{x}<\epsilon\}}\bigr)$. This means that $\widetilde{X}_t=N_t^{\epsilon}+\widehat{\sigma}(\epsilon)\widehat{X}_t^{\epsilon}$ where $N_t^\epsilon=I_1\bigl(h_t1_{[0,t]\times\{\abs{x}\geq\epsilon\}}\bigr)$ has finitely many jumps. Let $\widehat{W}$ be an isonormal Gaussian process with covariance structure given by $E\bigl[\widehat{W}(f)\widehat{W}(g)\bigr]=\int_{\RR^+}f(s)g(s)ds$.
\begin{theorem}\label{smalljumpsconv}
For a fixed $t$, $\widetilde{X}_t^{\epsilon}\stackrel{law}{\longrightarrow}\widetilde{Z}\sim\mathcal{N}(0,1)$ as $\epsilon\rightarrow0$ if
\begin{align}\label{smalljumpsconv0}
\frac{\dint_{[0,t]\times\{\abs{x}<\epsilon\}}\abs{xh_t(s,x)}^3d\nu(x)ds}{\widetilde{\sigma}_t(\epsilon)^3}\xrightarrow[\epsilon\rightarrow0]{}0
\end{align}
Moreover, suppose that $h_t(s,x)=h_t(s)$. Then $\widehat{X}_t^{\epsilon}\stackrel{law}{\longrightarrow}\widehat{W}(h_t)$ if
\begin{align}\label{smalljumpsconv1}
\frac{\int_{\{\abs{x}<\epsilon\}}\abs{x}^3d\nu(x)}{\widehat{\sigma}(\epsilon)^3}\xrightarrow[\epsilon\rightarrow0]{}0
\end{align}
\end{theorem}
\begin{proof}
To prove the first statement, it is enough to verify that the upper bound (\ref{poissoncase}) goes to zero as $\epsilon$ goes to zero. Notice that
$$\Lqnorm{\boldsymbol{D}X_t^{\epsilon}}{\mu}^2=\Lqnorm{\widetilde{\sigma}_t(\epsilon)^{-1}h_t1_{[0,t]\times\{\abs{x}\leq \epsilon\}}}{\mu}^2=1$$
so $E\abs{1-\Lqprod{\boldsymbol{D}X_t^{\epsilon}}{-\boldsymbol{D}L^{-1}X_t^{\epsilon}}{\mu}}=\abs{1-\Lqnorm{\boldsymbol{D}X_t^{\epsilon}}{\mu}^2}=0$. On the other hand,
$$E\left[\Lqprod{\abs{x\bigl(\boldsymbol{D}\widetilde{X}_t^{\epsilon}\bigr)^2}}{\abs{\boldsymbol{D}L^{-1}\widetilde{X}_t^{\epsilon}}}{\mu}\right]=E\left[\Lqprod{\abs{x}}{\abs{\boldsymbol{D}\widetilde{X}_t^{\epsilon}}^3}{\mu}\right]=\frac{\dint_{[0,t]\times\{\abs{x}<\epsilon\}}\abs{xh_t(s,x)}^3d\nu(x)ds}{\widetilde{\sigma}_t(\epsilon)^3}$$
So bound (\ref{poissoncase}) goes to zero as $\epsilon$ goes to zero if (\ref{smalljumpsconv0}) is true.
To prove the second statement, it is necessary to prove that for any times $\{t_1,\dots,t_d\}$ the random vector $\bigl(\widehat{X}_{t_1}^{\epsilon},\dots,\widehat{X}_{t_d}^{\epsilon}\bigr)\stackrel{law}{\longrightarrow}\bigl(\widehat{W}(h_{t_1}),\dots,\widehat{W}(h_{t_d})\bigr)$. This is very easy to achieve thanks to Lemma \ref{multdimcond} from remark \ref{multidimensional}, because only two conditions need to be checked. Notice that for all $\epsilon$
\begin{align*}
\hspace{-2cm}K_\epsilon(i,j)=&E\bigl[\widehat{X}_{t_i}^{\epsilon}\widehat{X}_{t_j}^{\epsilon}\bigr]=\Lqprod{\widehat{\sigma}(\epsilon)^{-1}h_{t_i}1_{[0,t_i]\times\{\abs{x}<\epsilon\}}}{\widehat{\sigma}(\epsilon)^{-1}h_{t_j}1_{[0,t_j]\times\{\abs{x}<\epsilon\}}}{\mu}\\
=&\Lqprod{h_{t_i}1_{[0,t_i]}}{h_{t_j}1_{[0,t_j]}}{}=E\bigl[\widehat{W}(h_{t_i}1_{[0,t_i]})\widehat{W}(h_{t_j}1_{[0,t_j]})\bigr]=C(i,j)
\end{align*}
so trivially $K_\epsilon(i,j)\rightarrow C(i,j)$ as $\epsilon\to\infty$. And for the second condition we have that
$$\Lqnorm{\abs{x}^{\frac{1}{2}}\abs{\widehat{\sigma}(\epsilon)^{-1}h_{t_i}1_{[0,t_i]\times\{\abs{x}<\epsilon\}}}^{\frac{3}{2}}}{\mu}^2=\int_0^{t_i}\abs{h_{s}}^3ds\underbrace{\frac{\int_{\{\abs{x}<\epsilon\}}\abs{x}^3d\nu(x)}{\widehat{\sigma}(\epsilon)^3}}_{\rightarrow0 \text{ by (\ref{smalljumpsconv1})}}\xrightarrow[\epsilon\rightarrow0]{}0$$
Since both conditions were fulfilled, the conclusion of the theorem follows.
\end{proof}
\begin{preremark}
Notice that if $h_t(s,x)=h_t(s)$ then
$$\widetilde{\sigma}_t(\epsilon)^2=\dint_{[0,t]\times\{\abs{x}<\epsilon\}}h_t(s)^2x^2d\nu(x)ds=\int_0^th_t(s)^2ds\int_{\{\abs{x}<\epsilon\}}x^2d\nu(x)=\int_0^th_t(s)^2ds\widehat{\sigma}(\epsilon)^2$$
so it immediately follows that (\ref{smalljumpsconv0}) is true if and only if (\ref{smalljumpsconv1}) is true. Assuming then, that (\ref{smalljumpsconv1}) is true, it may appear by the first part of Theorem \ref{smalljumpsconv}, that we can conclude the second statement, i.e., $\bigl(\int_0^th_t(s)^2ds\bigr)^{-\frac{1}{2}}\widehat{X}_t^{\epsilon}=\widetilde{X}_t^{\epsilon}\stackrel{law}{\longrightarrow}\widetilde{Z}\sim\mathcal{N}(0,1)$ as $\epsilon\rightarrow0$ which is equivalent to $\widehat{X}_t^{\epsilon}\stackrel{law}{\longrightarrow}\widehat{Z}\sim\mathcal{N}\bigl(0,\int_0^th_t(s)^2ds\bigr)$. Also, since $E\bigl[\widehat{X}_t^{\epsilon}\widehat{X}_s^{\epsilon}\bigr]=E\bigl[\widehat{W}(h_t1_{[0,t]})\widehat{W}(h_s1_{[0,s]})\bigr]$ for all $\epsilon$, $\widehat{X}_t^{\epsilon}$ and $\widehat{W}(h_t1_{[0,t]})$ coincide in their covariance structures. But this is not enough to prove that $\widehat{X}_t^{\epsilon}\stackrel{law}{\longrightarrow}\widehat{W}(h_t)$; it only proves that, marginally (and for fixed $t$), $\widehat{X}_t^{\epsilon}$ converges to a normal random variable, which does not imply that the joint distribution for all times is normally distributed. Then the conclusion does not follows from there. Hence the neccessity of Lemma \ref{multdimcond}.
\end{preremark}

\subsubsection{Example: Fractional L$\acute{\text{e}}$vy Process}

Condition (\ref{smalljumpsconv1}) is quite easy to be verified. In fact, the measure $d\nu(x)=\abs{x}^{-(2+\delta)}1_{\{-a\leq x\leq b\}}dx$ for $\delta\in(-1,1)$ and $a,b>0$ (no jumps bigger than $b$ or smaller than $-a$) is such that (\ref{smalljumpsconv1}) is fulfilled. To check this, note that $\int_{\{\abs{x}<\epsilon\}}\abs{x}^3d\nu(x)=\frac{2\epsilon^{2-\delta}}{(2-\delta)}$ and $\widehat{\sigma}(\epsilon)^3=\bigl(\frac{2\epsilon^{1-\delta}}{(1-\delta)}\bigr)^{\frac{3}{2}}$, so $\widehat{\sigma}(\epsilon)^{-3}\int_{\{\abs{x}<\epsilon\}}\abs{x}^3d\nu(x)=O(\epsilon^{\frac{1+\delta}{2}})$ and (\ref{smalljumpsconv1}) is true. So, for the example, assume that the measure $\nu$ is such that (\ref{smalljumpsconv1}) holds.

\begin{enumerate}

\item {\bf Fractional L$\acute{\text{e}}$vy Process (FLP):}\\
There are two ways to represent a fractional Brownian motion as an integral of a kernel with respect to Brownian motion, and both deliver the same process (see \cite{Celine} for a thorough explanation). One is the so-called Mandelbrot-Van Ness representation which is an integral over the whole real line with respect to a two sided Brownian motion. The other, is the so-called Molchan-Golosov representation which is an integral over a compact interval. In the L$\acute{\text{e}}$vy case, it is proved in \cite{Heikki} that these representations delivers different processes with very different characteristics. Because of this ``non-uniqueness'', the FLP generated by the Mandelbrot-Van Ness representation is called FLPMvN, and the one generated by the Molchan-Golosov representation is called FLPMG. It is known that FLP's have the same covariance structure as FBM. The advantage of FLPMvN over FLPMG is that the former is stationary and the latter is not (in general), as is shown in \cite{Heikki}. Nevertheless, since FLPMG is derived on a compact interval, Malliavin calculus can be applied to it.

Consider $X_t$ as an FLPMG, that is, $X_t=I_1(K^H_t)$ where $$\Lprod{K^H_t}{K^H_s}{}=\frac{1}{2}\bigl(\abs{t}^{2H}+\abs{s}^{2H}-\abs{t-s}^{2H}\bigr)$$
According to Theorem \ref{smalljumpsconv}, since (\ref{smalljumpsconv1}) is true and $h_t(s,x)=K^H_t(s)$, it follows that $X_t^{\epsilon}\stackrel{law}{\longrightarrow}\widehat{W}(K^H_t)$ as $\epsilon \to \infty$. But $\widehat{W}(K^H_t)=B^H_t$ is a fractional Brownian motion. In this case, we conclude that in order to simulate the paths of an FLPMG $X_t$, we just need to fix $\epsilon$ small enough, and simulate the finitely many jumps part $N_t^{\epsilon}=I_1(K^H_t1_{\{\abs{x}\geq\epsilon\}})$ along with an (independent) FBM part $B^H_t=\widehat{W}(K^H_t)$, because $X_t\stackrel{\text{law}}{\approx}N_t^{\epsilon}+\widehat{\sigma}(\epsilon)B^H_t$.

\end{enumerate}

\subsection{The Wiener-Poisson space case:\\ Product of O-U processes}

Finally, the second order Poincar$\acute{\text{e}}$ inequality developed in the combined space is used to obtain a CLT for mixed processes. First, notice that if we have a double Ornstein-Uhlenbeck (O-U) process as a sum of a Wiener O-U process $Y_t$ plus a Poisson O-U process $Z_t$ (independent of $Y_t$), it can be proved (in two different ways) that the functional $F_T=T^{-\frac{1}{2}}\int_0^TY_t+Z_tdt$ converges to a normal random variable as $T\to\infty$. The first way is to seperate $F_T$ to the terms $T^{-\frac{1}{2}}\int_0^TY_tdt$ and $T^{-\frac{1}{2}}\int_0^TZ_tdt$ and use the inequalities (\ref{wienercase}) and (\ref{poissoncase}) respectively to prove that each part goes to a normal. The other method is to use inequality (\ref{mainineq1}) and just do one computation. The second way is clearly faster (since the kernels are the same). This is one advantage of having the inequality in the combined space.

\bigskip

For our example, lets focus on a process for which we cannot use either of the inequalities, (\ref{wienercase}) nor (\ref{poissoncase}), to prove a CLT. First, for simplicity's sake (to avoid dealing with constants), assume that the triplet for the underlying L$\acute{\text{e}}$vy process is given by $(0,1,\nu)$, where $\int_\RR x^2d\nu(x)=1$. Moreover, assume that $\int_\RR \abs{x}^3d\nu(x)<\infty$ and $\int_{\RR_0}x^4d\nu(x)<\infty$. Let $Y_t=\int_0^t\sqrt{2\lambda}e^{-\lambda(t-s)}dW_s$ and  $Z_t=\dint_{[0,t]\times\RR_0}\sqrt{2\lambda}e^{-\lambda(t-s)}xd\widetilde{N}(s,x)$, so, $Y_t$ is a Wiener O-U process and $Z_t$ is a Poisson (pure jump) O-U process. If $h_t(s)=\sqrt{2\lambda}e^{-\lambda(t-s)}$ then the double O-U process mentioned above is just $Y_t+Z_t=I_1(h_t)$. Now, define $h_t^{(0)}(s,x)=h_t(s)1_{\{x=0\}}(x)$ and $h_t^{(1)}(s,x)=h_t(s)1_{\{x\neq0\}}(x)$, then $Y_t=I_1\bigl(h_t^{(0)}\bigr)$ and $Z_t=I_1\bigl(h_t^{(1)}\bigr)$. Notice that due to the normalization of the L$\acute{\text{e}}$vy triplet we have that $C(t,s)=\Lqprod{h_t^{(0)}}{h_s^{(0)}}{\mu}=\Lqprod{h_t^{(1)}}{h_s^{(1)}}{\mu}$. The goal of this subsection is to prove that $F_T=T^{-\frac{1}{2}}\int_0^TY_tZ_tdt\stackrel{law}{\longrightarrow}Z\sim\mathcal{N}(0,\Sigma^2)$ as $T\to\infty$.

\bigskip

Since $h_t^{(0)}$ and $h_t^{(1)}$ have disjoint supports $\bigl($and using the product formula (\ref{prodint})$\bigr)$, $Y_tZ_t=I_2\bigl(h_t^{(0)}\widetilde{\otimes}h_t^{(1)}\bigr)$, and by Fubini $F_T=I_2\bigl(T^{-\frac{1}{2}}\int_{\cdot\vee\cdot}^Th_t^{(0)}\widetilde{\otimes}h_t^{(1)}dt\bigr)$. Hence $F_T$ lies in the $2^{\text{nd}}$ chaos. According to Corollary \ref{mainqchaos} we just need to check conditions (\ref{2cond1}), (\ref{2cond2}), (\ref{2cond3}), (\ref{2cond4}) and (\ref{2cond5}).
\begin{itemize}

\item {\bf Expectation of the First Derivative's Norm:}\\
\begin{itemize}

\item First Malliavin Derivative:
$$\boldsymbol{D}_zF_T=T^{-\frac{1}{2}}\int_0^TI_1\bigl(h_t^{(0)}\bigr)h_t^{(1)}(z)+I_1\bigl(h_t^{(1)}\bigr)h_t^{(0)}(z)dt$$

\item Norm of the First Malliavin Derivative:
\begin{align*}
\Lqnorm{\boldsymbol{D}_zF_T}{\mu}^2=&T^{-1}\int_{[0,T]^2}\Lqprod{I_1\bigl(h_t^{(0)}\bigr)h_t^{(1)}(z)+I_1\bigl(h_t^{(1)}\bigr)h_t^{(0)}(z)}{I_1\bigl(h_s^{(0)}\bigr)h_s^{(1)}(z)+I_1\bigl(h_s^{(1)}\bigr)h_s^{(0)}(z)}{\mu}dtds\\
=&T^{-1}\int_{[0,T]^2}I_1\bigl(h_t^{(0)}\bigr)I_1\bigl(h_s^{(0)}\bigr)\Lqprod{h_t^{(1)}}{h_s^{(1)}}{\mu}+I_1\bigl(h_t^{(1)}\bigr)I_1\bigl(h_s^{(1)}\bigr)\Lqprod{h_t^{(0)}}{h_s^{(0)}}{\mu}dtds\\
=&T^{-1}\int_{[0,T]^2}\bigl(I_1\bigl(h_t^{(0)}\bigr)I_1\bigl(h_s^{(0)}\bigr)+I_1\bigl(h_t^{(1)}\bigr)I_1\bigl(h_s^{(1)}\bigr)\bigr)C(t,s)dtds
\end{align*}
then,
\begin{align*}
\Lqnorm{\boldsymbol{D}_zF_T}{\mu}^4\leq& 2T^{-2}\biggl[\biggl(\int_{[0,T]^2}I_1\bigl(h_t^{(0)}\bigr)I_1\bigl(h_s^{(0)}\bigr)C(t,s)dtds\biggr)^2+\biggl(\int_{[0,T]^2}I_1\bigl(h_t^{(1)}\bigr)I_1\bigl(h_s^{(1)}\bigr)C(t,s)dtds\biggr)^2\biggr]\\
=&2T^{-2}\int_{[0,T]^4}\biggl[\prod_{i=1}^4I_1\bigl(h_{t_i}^{(0)}\bigr)+\prod_{i=1}^4I_1\bigl(h_{t_i}^{(1)}\bigr)\biggr]C(t_1,t_2)C(t_3,t_4)d\vec{t}
\end{align*}

\item Expectation of the First Derivative's Norm:\\
Notice that by the product formula (\ref{prodint}) we have that
$$\prod_{i=1}^4I_1\bigl(h_{t_i}^{(0)}\bigr)=\biggl[\Lqprod{h_{t_1}^{(0)}}{h_{t_2}^{(0)}}{\mu}+I_2\bigl(h_{t_1}^{(0)}\widetilde{\otimes}h_{t_2}^{(0)}\bigr)\biggr]\biggl[\Lqprod{h_{t_3}^{(0)}}{h_{t_4}^{(0)}}{\mu}+I_2\bigl(h_{t_3}^{(0)}\widetilde{\otimes}h_{t_4}^{(0)}\bigr)\biggr]$$
and
$$ \prod_{i=1}^4I_1\bigl(h_{t_i}^{(1)}\bigr)=\biggl[\Lqprod{h_{t_1}^{(1)}}{h_{t_2}^{(1)}}{\mu}+I_1\bigl(h_{t_1}^{(1)}\otimes_0^1h_{t_2}^{(1)}\bigr)+I_2\bigl(h_{t_1}^{(1)}\widetilde{\otimes}h_{t_2}^{(1)}\bigr)\biggr]\biggl[\Lqprod{h_{t_3}^{(1)}}{h_{t_4}^{(1)}}{\mu}+I_1\bigl(h_{t_3}^{(1)}\otimes_0^1h_{t_4}^{(1)}\bigr)+I_2\bigl(h_{t_3}^{(1)}\widetilde{\otimes}h_{t_4}^{(1)}\bigr)\biggr]$$
so,
$$E\biggl[\prod_{i=1}^4I_1\bigl(h_{t_i}^{(0)}\bigr)\biggr]=C(t_1,t_2)C(t_3,t_4)+\Lprod{h_{t_1}^{(0)}\widetilde{\otimes}h_{t_2}^{(0)}}{h_{t_3}^{(0)}\widetilde{\otimes}h_{t_4}^{(0)}}{\mu}$$
and
$$E\biggl[\prod_{i=1}^4I_1\bigl(h_{t_i}^{(1)}\bigr)\biggr]=C(t_1,t_2)C(t_3,t_4)+\Lprod{h_{t_1}^{(1)}\otimes_0^1h_{t_2}^{(1)}}{h_{t_3}^{(1)}\otimes_0^1h_{t_4}^{(1)}}{\mu}+\Lprod{h_{t_1}^{(1)}\widetilde{\otimes}h_{t_2}^{(1)}}{h_{t_3}^{(1)}\widetilde{\otimes}h_{t_4}^{(1)}}{\mu}$$
Also notice that,
$$C(t,s)=\int_0^{t\wedge s}2\lambda e^{-\lambda(t+s-2u)}du=e^{-\lambda\abs{t-s}}-e^{-\lambda(t+s)}\leq e^{-\lambda\abs{t-s}}\leq1$$
$$\underbrace{\Lprod{h_{t_1}^{(1)}\otimes_0^1h_{t_2}^{(1)}}{h_{t_3}^{(1)}\otimes_0^1h_{t_4}^{(1)}}{\mu}}_{\Lprod{xh_{t_1}^{(1)}h_{t_2}^{(1)}}{xh_{t_3}^{(1)}h_{t_4}^{(1)}}{\mu}}=\int_0^{\min\{t_1,t_2,t_3,t_4\}}\int_{\RR_0}x^44\lambda^2e^{-\lambda(t_1+t_2+t_3+t_4-4u)}d\nu(x)du\leq\lambda\int_{\RR_0}x^4d\nu(x)$$
$$\Lprod{h_{t_1}^{(0)}\widetilde{\otimes}h_{t_2}^{(0)}}{h_{t_3}^{(0)}\widetilde{\otimes}h_{t_4}^{(0)}}{\mu^{\otimes2}}=\Lprod{h_{t_1}^{(1)}\widetilde{\otimes}h_{t_2}^{(1)}}{h_{t_3}^{(1)}\widetilde{\otimes}h_{t_4}^{(1)}}{\mu^{\otimes2}}=\frac{C(t_1,t_3)C(t_2,t_4)+C(t_1,t_4)C(t_2,t_3)}{2}\leq1$$
Putting all this together we get,
\begin{align*}
E\bigl[\Lqnorm{\boldsymbol{D}F_T}{\mu}^4\bigr]=&2T^{-2}\int_{[0,T]^4}\biggl(E\biggl[\prod_{i=1}^4I_1\bigl(h_{t_i}^{(0)}\bigr)\biggr]+E\biggl[\prod_{i=1}^4I_1\bigl(h_{t_i}^{(1)}\bigr)\biggr]\biggr)C(t_1,t_2)C(t_3,t_4)d\vec{t}\\
\leq&2\biggl(4+\lambda\int_{\RR_0}x^4d\nu(x)\biggr)\biggl(T^{-1}\int_{[0,T]^2}C(t,s)dtds\biggr)^2\\
\leq&2\biggl(4+\lambda\int_{\RR_0}x^4d\nu(x)\biggr)\biggl(T^{-1}\int_{[0,T]^2}e^{-\lambda\abs{t-s}}dtds\biggr)^2\\
=&2\biggl(4+\lambda\int_{\RR_0}x^4d\nu(x)\biggr)\biggl(2T^{-1}\int_0^T\int_0^te^{-\lambda(t-s)}dsdt\biggr)^2\\
\leq&2\biggl(4+\lambda\int_{\RR_0}x^4d\nu(x)\biggr)\biggl(\frac{2}{\lambda}\biggr)^2
\end{align*}
\end{itemize}
All this proves that
\begin{align*}
\boxed{E\bigl[\Lqnorm{\boldsymbol{D}F_T}{\mu}^4\bigr]^{\frac{1}{4}}=O(1)\ \text{ as }\ T\to\infty}
\end{align*}

\item {\bf Expectation of the Cube of the First Derivative's Norm:}
\begin{itemize}

\item Cube of the First Malliavin Derivative:\\
Since $h_{t}^{(0)}(z)\cdot h_{s}^{(1)}(z)=0$ for all $z\in\RR^+\times\RR$ then,
$$\abs{\boldsymbol{D}F_T}^3=\abs{T^{-\frac{3}{2}}\int_{[0,T]^3}\biggl[\prod_{i=1}^3I_1\bigl(h_{t_i}^{(0)}\bigr)h_{t_i}^{(1)}+\prod_{i=1}^3I_1\bigl(h_{t_i}^{(1)}\bigr)h_{t_i}^{(0)}\biggr]d\vec{t}}$$

\item Norm of the Cube of the First Malliavin Derivative:\\
Since $x\cdot h_{t}^{(0)}(z)=0$ for all $z=(t,x)\in\RR^+\times\RR$ then,
$$\Lqprod{\abs{x}}{\abs{\boldsymbol{D}F_T}^3}{\mu}\leq T^{-\frac{3}{2}}\int_{[0,T]^3}\prod_{i=1}^3\abs{I_1\bigl(h_{t_i}^{(0)}\bigr)}\Lqprod{\abs{x}}{\prod_{i=1}^3\abs{h_{t_i}^{(1)}}}{\mu}d\vec{t}$$

\item Expectation of the Cube of the First Derivative's Norm:\\
Notice that by H$\ddot{\text{o}}$lder we have,
\begin{align*}
E\biggl[\prod_{i=1}^3\abs{I_1\bigl(h_{t_i}^{(0)}\bigr)}\biggr]\leq& E\bigl[\bigl(\hspace{-.7cm}\overbrace{I_1\bigl(h_{t_1}^{(0)}\bigr)I_1\bigl(h_{t_2}^{(0)}\bigr)}^{\Lqprod{h_{t_1}^{(0)}}{h_{t_2}^{(0)}}{\mu}+I_2\bigl(h_{t_1}^{(0)}\widetilde{\otimes}h_{t_2}^{(0)}\bigr)}\hspace{-.7cm}\bigr)^2\bigr]^{\frac{1}{2}}E\bigl[\bigl(I_1\bigl(h_{t_3}^{(0)}\bigr)\bigr)^2\bigr]^{\frac{1}{2}}\\
\leq& \biggl(\abs{C(t_1,t_2)}+E\bigl[\bigl(I_2\bigl(h_{t_1}^{(0)}\widetilde{\otimes}h_{t_2}^{(0)}\bigr)\bigr)^2\bigr]^{\frac{1}{2}}\biggr)E\bigl[\bigl(I_1\bigl(h_{t_3}^{(0)}\bigr)\bigr)^2\bigr]^{\frac{1}{2}}\\
\leq&\biggl(\abs{C(t_1,t_2)}+\Lqnorm{h_{t_1}^{(0)}\otimes h_{t_2}^{(0)}}{\mu^{\otimes2}}\biggr)\Lqnorm{h_{t_3}^{(0)}}{\mu}\leq2\\
\end{align*}
and
\begin{align*}
\Lqprod{\abs{x}}{\prod_{i=1}^3\abs{h_{t_i}^{(1)}}}{\mu}=&\int_0^{\min\{t_1,t_2,t_3\}}(2\lambda)^{\frac{3}{2}}e^{-\lambda(t_1+t_2+t_3-3u)}du\int_{\RR_0}\abs{x}^3d\nu(x)\\
\leq&\frac{2\sqrt{2\lambda}}{3}\biggl(\int_{\RR_0}\abs{x}^3d\nu(x)\biggr)e^{-\lambda(t_1+t_+t_3-3\min\{t_1,t_2,t_3\})}
\end{align*}
Putting all together we get,
\begin{align*}
E\bigl[\Lqprod{\abs{x}}{\abs{\boldsymbol{D}F_T}^3}{\mu}\bigr]\leq& T^{-\frac{3}{2}}\int_{[0,T]^3}E\biggl[\prod_{i=1}^3\abs{I_1\bigl(h_{t_i}^{(0)}\bigr)}\biggr]\Lqprod{\abs{x}}{\prod_{i=1}^3\abs{h_{t_i}^{(1)}}}{\mu}d\vec{t}\\
\leq& \frac{4\sqrt{2\lambda}}{3}\biggl(\int_{\RR_0}\abs{x}^3d\nu(x)\biggr)T^{-\frac{3}{2}}\int_{[0,T]^3}e^{-\lambda(t_1+t_+t_3-3\min\{t_1,t_2,t_3\})}d\vec{t}\\
=&\frac{24\sqrt{2\lambda}}{3}\biggl(\int_{\RR_0}\abs{x}^3d\nu(x)\biggr)T^{-\frac{3}{2}}\int_0^T\int_0^t\int_0^se^{-\lambda(t+s-2u)}dudsdt\\
\leq& \frac{4\sqrt{2}\bigl(\int_{\RR_0}\abs{x}^3d\nu(x)\bigr)}{\lambda^{\frac{3}{2}}\sqrt{T}}=O(T^{-\frac{1}{2}})
\end{align*}

\end{itemize}
All this proves that
\begin{align*}
\boxed{E\bigl[\Lqprod{\abs{x}}{\abs{\boldsymbol{D}F}^3}{\mu}\bigr]\rightarrow0\ \text{ as } \ T\rightarrow\infty}
\end{align*}

\item {\bf Expectation of the Contraction's Norm:}
\begin{itemize}

\item Second Malliavin Derivative:
$$\boldsymbol{D}^2_{z_1,z_2}F_T=T^{-\frac{1}{2}}\int_0^Th_{t}^{(1)}(z_1)h_{t}^{(0)}(z_2)+h_{t}^{(0)}(z_1)h_{t}^{(1)}(z_2)dt$$

\item Contraction of order 1:
\begin{align*}
\boldsymbol{D}^2F_T\otimes_1\boldsymbol{D}^2F_T=&T^{-1}\int_{[0,T]}h_{t}^{(1)}(z_1)h_{s}^{(1)}(z_1)\overbrace{\Lqprod{h_{t}^{(0)}}{h_{s}^{(0)}}{\mu}}^{C(t,s)\leq1}+h_{t}^{(0)}(z_1)h_{s}^{(0)}(z_1)\overbrace{\Lqprod{h_{t}^{(1)}}{h_{s}^{(1)}}{\mu}}^{C(t,s)\leq1}dtds\\
\leq&T^{-1}\int_{[0,T]}h_{t}^{(1)}(z_1)h_{s}^{(1)}(z_1)+h_{t}^{(0)}(z_1)h_{s}^{(0)}(z_1)dtds
\end{align*}

\item Norm of the Contraction:
$$\Lqnorm{\boldsymbol{D}^2F_T\otimes_1\boldsymbol{D}^2F_T}{\mu}^2\leq T^{-2}\int_{[0,T]^4}\Lqprod{h_{t_1}^{(0)}h_{t_2}^{(0)}}{h_{t_3}^{(0)}h_{t_4}^{(0)}}{\mu}+\Lqprod{h_{t_1}^{(1)}h_{t_2}^{(1)}}{h_{t_3}^{(1)}h_{t_4}^{(1)}}{\mu}d\vec{t}$$

\item Expectation of the Contraction's Norm:\\
Notice that
$$\Lqprod{h_{t_1}^{(0)}h_{t_2}^{(0)}}{h_{t_3}^{(0)}h_{t_4}^{(0)}}{\mu}=\Lqprod{h_{t_1}^{(1)}h_{t_2}^{(1)}}{h_{t_3}^{(1)}h_{t_4}^{(1)}}{\mu}=\lambda \bigl(e^{-\lambda(t_1+t_2+t_3+t_4-4\min\{t_1,t_2,t_3,t_4\})}-e^{-\lambda(t_1+t_2+t_3+t_4)}\bigr)$$
so,
$$\int_{[0,T]^4}\Lqprod{h_{t_1}^{(i)}h_{t_2}^{(i)}}{h_{t_3}^{(i)}h_{t_4}^{(i)}}{\mu}\leq24\lambda\int_0^T\int_0^t\int_0^s\int_0^ue^{-\lambda(t+s+u-3v)}dvdudsdt\leq \frac{4T}{\lambda^2}$$
Putting all together we get,
$$E\bigl[\Lqnorm{\boldsymbol{D}^2F_T\otimes_1\boldsymbol{D}^2F_T}{\mu}^2\bigr]\leq 2T^{-2}\int_{[0,T]^4}\Lqprod{h_{t_1}^{(i)}h_{t_2}^{(i)}}{h_{t_3}^{(i)}h_{t_4}^{(i)}}{\mu}d\vec{t}\leq T^{-2}\frac{8T}{\lambda^2}=\frac{8}{T\lambda^2}$$

\end{itemize}
All this proves that,
\begin{align*}
\boxed{E\bigl[\Lqnorm{\boldsymbol{D}^2F_T\otimes_1\boldsymbol{D}^2F_T}{\mu}^2\bigr]\rightarrow0\ \text{ as } \ T\rightarrow\infty}
\end{align*}

\item {\bf Expectation of the Squared Second Derivative's Norm:}
\begin{itemize}

\item Square of the Second Malliavin Derivative:
$$(\boldsymbol{D}^2_{z_1,z_2}F_T)^2=T^{-1}\int_{[0,T]^2}h_t^{(0)}(z_1)h_s^{(0)}(z_1)h_t^{(1)}(z_2)h_s^{(1)}(z_2)+h_t^{(1)}(z_1)h_s^{(1)}(z_1)h_t^{(0)}(z_2)h_s^{(0)}(z_2)dtds$$

\item Inner Product of the Squared Second Malliavin Derivative:\\
Samely as above, $x\cdot h_t^{(0)}=0$, so
$$ \Lqprod{x}{(\boldsymbol{D}^2F_T)^2}{\mu}=T^{-1}\int_{[0,T]^2}h_t^{(0)}h_s^{(0)}\overbrace{\Lqprod{x}{h_t^{(1)}h_s^{(1)}}{\mu}}^{\int_{\RR_0}x^3d\nu(x)C(t,s)}dtds\leq \int_{\RR_0}\abs{x}^3d\nu(x)T^{-1}\int_{[0,T]^2}h_t^{(0)}h_s^{(0)}dtds$$

\item Expectation of the Squared Second Derivative's Norm:\\
By the above computations we have,
$$E\left[\Lqnorm{\Lqprod{x}{(\boldsymbol{D}^2F_T)^2}{\mu}}{\mu}^2\right]\leq\biggl(\int_{\RR_0}\abs{x}^3d\nu(x)\biggr)^2T^{-2}\int_{[0,T]^4}\Lqprod{h_{t_1}^{(0)}h_{t_2}^{(0)}}{h_{t_3}^{(0)}h_{t_4}^{(0)}}{\mu}d\vec{t}\leq\frac{4}{\lambda^2T}$$

\end{itemize}
All this proves that,
\begin{align*}
\boxed{E\biggl[\Lqnorm{\Lqprod{x}{(\boldsymbol{D}^2F)^2}{\mu}}{\mu}^2\biggr]\rightarrow0\ \text{ as } \ T\rightarrow\infty}
\end{align*}

\item {\bf Existence of the Variance:}\\
Since $F_T=I_2\bigl(T^{-\frac{1}{2}}\int_{\cdot\vee\cdot}^Th_t^{(0)}\widetilde{\otimes}h_t^{(1)}dt\bigr)$ then \begin{align*}
\text{Var}[F_T]=&\Lqnorm{T^{-\frac{1}{2}}\int_{\cdot\vee\cdot}^Th_t^{(0)}\widetilde{\otimes}h_t^{(1)}dt}{\mu^{\otimes2}}^2=T^{-1}\int_{[0,T]^2}\biggl(\int_{s_1\vee s_2}^T2\lambda e^{-\lambda(2t-s_1-s_2)}dt\biggr)^2ds_1ds_2\\
=& T^{-1}\int_{[0,T]^2}\bigl(e^{-\lambda\abs{s_1-s_2}}-e^{-\lambda(2T-s_1-s_2)}\bigr)^2ds_1ds_2\\
=& T^{-1}\int_{[0,T]^2}e^{-2\lambda\abs{s_1-s_2}}-2e^{-2\lambda(T-s_1\wedge s_2)}+e^{-2\lambda(2T-s_1-s_2)}ds_1ds_2\\
=& T^{-1}\biggl(\frac{T}{\lambda}-6\frac{(1-e^{-2\lambda T})}{4\lambda^2}+4Te^{-2\lambda T}+\frac{(1-e^{-2\lambda T})^2}{4\lambda^2}\biggr)=\frac{1}{\lambda}+O(T^{-1})
\end{align*}
All this proves that,
\begin{align*}
\boxed{\text{Var}[F_T]\rightarrow \frac{1}{\lambda}\in(0,\infty)\ \text{ exists as }\ T\to\infty}
\end{align*}

\end{itemize}
Since all five conditions are met, by Corollary \ref{mainqchaos} we have that $F_T\stackrel{law}{\longrightarrow}N\sim\mathcal{N}(0,\frac{1}{\lambda})$ as $T\to\infty$. Moreover, due to the quantitative property of the inequality, is possible to estimate (from the computations above) that the rate of convergence to normality is at least $O\bigl(T^{-\frac{1}{4}}\bigr)$, i.e., $d_W\bigl(\frac{F_T}{\sqrt{\text{Var}[F_T]}},N\bigr)=O\bigl(T^{-\frac{1}{4}}\bigr)$ as $T\to\infty$. This rate is similar to the one obtained for the linear functionals of Gaussian-subordinated fields with underlying process given by the increments of FBM or the fractional-driven O-U, when $H\in\bigl(0,\frac{1}{2}\bigr)$. I believe that the right rate of convergence for these processes should be $O\bigl(T^{-\frac{1}{2}}\bigr)$!

\end{document}